\documentclass[11pt,reqno,oneside]{amsart}

\usepackage[utf8]{inputenc}
\usepackage[english]{babel}
\usepackage[a4paper, total={6in, 9in}]{geometry}
\usepackage{pgf,tikz,pgfplots} \pgfplotsset{compat=newest}
\usetikzlibrary{arrows}
\usetikzlibrary[patterns]
\usepackage{upgreek}
\usepackage{amsmath, amscd, amssymb, amsthm, mathrsfs}
\usepackage[abbrev,nobysame,alphabetic]{amsrefs}
\usepackage{xcolor}
\usepackage{cases}
\usepackage{hyperref}
\newtheorem{thm}{Theorem}[section]
\newtheorem*{thm*}{Theorem}
\newtheorem{lem}[thm]{Lemma}
\newtheorem{cor}[thm]{Corollary}
\newtheorem{prop}[thm]{Proposition}
\theoremstyle{definition}

\newtheorem{asm}{Assumption}[section]
\newtheorem{rem}{Remark}[section]
\newtheorem*{rem*}{Remark}
\numberwithin{equation}{section}

\usepackage[pagewise]{lineno}
\newcommand{\R}{\mathbb{R}}
\newcommand{\Rn}{{\mathbb{R}^n}}

\DeclareMathOperator{\dist}{dist}
\DeclareMathOperator{\divr}{div}
\DeclareMathOperator{\supp}{\mathrm{supp}}

\newcommand{\dif}[1]{\,\mathrm{d}{#1}}
\newcommand{\df}{\mathrm{d}}
\newcommand{\norm}[1]{\lVert #1 \rVert}

\allowdisplaybreaks

\title[Hyperbolic equation inverse stability]{Inverse stability for hyperbolic equations with different initial conditions}

\author{Shiqi Ma}
\address{School of Mathematics, Jilin University, Changchun 130012, China}
\email{mashiqi@jlu.edu.cn, mashiqi01@gmail.com}

\begin{document}

\begin{abstract}
	We establish Lipschitz stability for both the potential and the initial conditions from a single boundary measurement in the context of a hyperbolic boundary initial value problem. In our setting, the initial conditions are allowed to differ for different potentials. Compared to the traditional B-K method, our approach does not require the time reflection step. This advantage opens a route towards applying our method to the fixed angle inverse scattering problem, which remains unresolved for the single incident wave case. To achieve our result, we impose a pointwise comparison estimate on the Laplacian of the differences of the initial conditions, together with a positivity condition on one of the initial positions. The assumption generalizes previous stability results that usually assume the difference to be zero. We propose the initial-potential problem and establish a Lipschitz stability, then propose an iterative reconstruction scheme and show its geometric convergence. The initial-potential problem serves as an attempt to relate initial boundary value problem to the scattering problem, and to explore the possibility of relaxing the positivity requirement on the initial data. We also establish a new pointwise Carleman estimate, whose proof is significantly shorter and the reasoning is much clearer than existing proofs.
	
	\medskip
	
	
	\medskip
	
	
\end{abstract}

\maketitle

%\tableofcontents

\section{Introduction} \label{sec:intr-PR25}

Let $\Omega$ be a bounded domain in $\Rn$ with smooth boundary $\partial \Omega$.
Let $T > 0$ and denote $\square := \partial_t^2 - \Delta$ as the wave operator.
We investigate the following hyperbolic initial boundary value problem (IBVP),
\begin{equation} \label{eq:1-PR25}
	\left\{\begin{aligned}
		(\square + q(x)) u(t,x) & = 0, && (t,x) \in Q_T := (0,T) \times \Omega, \\
		u(t,x) & = g(t,x), && (t,x) \in \Sigma_T := (0,T) \times \partial \Omega, \\
		u(0,x) = a(x), \ \partial_t u(0,x) & = b(x), && x \in \Omega,
	\end{aligned}\right.
\end{equation}
under the following regularity settings:
\begin{equation} \label{eq:1con-PR25}
	q \in C^{s+1}(\overline \Omega), \ \ g \in H^{s+1/2}(\Sigma_T), \ \ a \in H^{s + 1}(\Omega), \ \ b \in H^{s}(\Omega), \ \ s \geq \max\{\lceil n/2 \rceil + 1, 3\},
\end{equation}
where $s$ is an integer.
The well-posedness of the system \eqref{eq:1-PR25}-\eqref{eq:1con-PR25} with the compatibility conditions is classical, see e.g.~\cite{EvanPDEs}*{Chapter 7}, \cite{llt1986non}, and \cite{Lions1972book}.

The inverse stability of hyperbolic equations has been extensively studied in the past decades.

Carleman estimates were introduced in the field of inverse problems for the first time by Bukhgeim and Klibanov in \cite{BukhgeimKlibanov}.
In that paper, the authors used Carleman estimates to prove a uniqueness result, while the stability version is summarized in \cite{Klibanov1992}.
We could briefly revise the method by reviewing \cite{Masahiro1999}.
Yamamoto in \cite{Masahiro1999} studied the system \eqref{eq:1-PR25}.
He indicated the connection between the inverse potential problem of recovering $q$ of \eqref{eq:1-PR25} and the inverse source problem of recovering $f$ of the following PDE
\begin{equation*}
	\left\{\begin{aligned}
		(\square + q) u & = f(x) R(t,x), && \text{in} \ Q_T, \\
		u & = 0, && \text{on} \ \Sigma_T, \\
		u(0,x) & = 0, \ \partial_t u(0,x) = 0, && x \in \Omega.
	\end{aligned}\right.
\end{equation*}
Indeed, when subtracting \eqref{eq:1-PR25} for $j=1$ and $j=2$ with $a_1 = a_2$, $b_1 = b_2$, we obtain the above PDE with $u = u_1 - u_2$, $q = q_1$, $f = q_2 - q_1$ and $R = u_2$.
By taking the time derivative of $u$, we further obtain
\begin{equation*}
	\left\{\begin{aligned}
		(\square + q) v & = f(x) \partial_t R(t,x), && \text{in} \ Q_T, \\
		v & = 0, && \text{on} \ \Sigma_T, \\
		v(0,x) & = 0, \ \partial_t v(0,x) = f(x) R(0,x), && x \in \Omega,
	\end{aligned}\right.
\end{equation*}
with $v = \partial_t u$.
Then, by combining a weighted energy estimate with a Carleman estimate, and assuming that $|R(0,x)| \geq C > 0$ for some fixed constant $C$, the norm of $f$ can be bounded by that of $\partial_\nu v$ on the lateral boundary, obtaining a Lipschitz stability of the form
\[
\norm{q_1 - q_2} \lesssim \norm{\partial_\nu (u_1 - u_2)|_{\Sigma_T}}.
\]
The condition $|R(0,x)| \geq C > 0$ is in fact 
\begin{equation} \label{eq:age0-PR25}
	|a(x)| \geq C > 0, \quad \forall x \in \Omega
\end{equation}
because $R = u_2$.
This method is sometimes referred to as the B-K method in recognition of Bukhgeim and Klibanov's contribution, and whether the condition \eqref{eq:age0-PR25} can be removed in this method still remains open.
In our first result Theorem \ref{thm:1-PR25}, we need to impose this condition.
And in Theorem \ref{thm:2-PR25} we see conditions like \eqref{eq:age0-PR25} can be lifted when the initial conditions and the potential are coupled.
Following the scheme of the B-K method, many studies have been conducted.
\cite{IY01} studied the homogeneous Neumann boundary data case, i.e., $\partial_\nu u = 0$.
\cite{Imanuvilov2001} proved a Lipschitz stable recovery for the potential using interior data $u_1 - u_2$ in $(0,T) \times \omega$ where $\omega \subset \Omega$ satisfies
\(
\{ x \in \partial \Omega \,;\, (x - x_0) \cdot \nu(x) \geq 0\} \subset \partial \omega
\)
for a fixed $x_0 \notin \overline \Omega$.
In \cites{IY01, Imanuvilov2001}, the dimension is limited to $n \leq 3$.
\cites{Be04, Klibanov2006Lipschitz} applied the B-K method to the sound speed case and obtained Lipschitz stability results.
Klibanov in his recent monograph
\cite{Klibanov2021book}*{Section 3.6} obtained a Lipschitz stability for the sound speed.

The recent works \cites{Klibanov2022point, Klibanov2023plane} applied this method to inverse scattering problems for the sound speed in which the hyperbolic system is driven by a point source.
They obtained H\"older stability estimates.
The B-K method has also been applied to other types of PDEs.
For example, it was used for parabolic equations in \cite{IY1998parabolic}, and for Schr\"odinger equations in \cites{Baudouin2002Schrodinger, Huang2020Stability}.
In the case of partial data, Isakov and Yamamoto \cite{Isakov2003Stability} established Lipschitz stability for an inverse source problem using measurements on the subboundary $(0,T) \times \Gamma$ where $\Gamma \subset \partial \Omega$.
\cite{IY2003Determination} established H\"older stability for the sound speed recovery using data on a subboundary satisfying some geometric conditions.
Bellassoued proved logarithmic stability for the potential in \cite{Bellassoued2004log}, using the lateral boundary data on $(0,T) \times \Gamma$ where $\Gamma$ is an arbitrarily small subset of $\partial \Omega$.
Then, this logarithmic stability was generalized to the sound speed case in \cite{Bellassoued2006acoustic} and to the Schr\"odinger equation in \cite{Bellassoued2009Schrodinger}.
There are other works in this field and we only listed a few.
For more references, we refer the readers to monographs \cites{Klibanov2012book, Bellassoued2017book, Klibanov2021book, Yamamoto2025book} and references therein.

As far as we know, all previous work for inverse stability of hyperbolic IBVP, except for \cite{Cipolatti2011finite}, assumes that $a_1 = a_2$ and $b_1 = b_2$, i.e.~the initial conditions are fixed.
In the article \cite{Cipolatti2011finite}, the authors considered the case where both the potential and initial conditions differ between $j=1$ and $j=2$, and recovered them all.
However, their approach requires multiple lateral boundary measurements as well as an additional internal observation of $u(t,x)$ at some $t = T_0 \in (0,T)$.

Besides the initial conditions issue, the B-K method needs a \textit{time reflection} step, which requires extending the solution $u(t,x)$ from $t > 0$ to $t < 0$ by $u(-t,x) := u(t,x)$,
see for example \cite{IY01}*{(3.8)}, \cite{Imanuvilov2001}*{(3.9)}, \cite{IY2003Determination}*{(3.4)}, \cite{Bellassoued2004log}*{(3.20)}, \cite{Huang2020Stability}*{after (3.1)}.
Otherwise, the Carleman estimates employed would introduce a boundary term at $t=0$, which would jeopardize the whole derivation.
By a time reflection, the boundary term moves from $t = 0$ to $t = -T$, then it can be handled well.
But, to perform the time reflection,
the initial conditions have to be the same for $u_1$ and $u_2$.

In this work, we lift both the time reflection restriction and the initial condition issue, and we recover the potential together with initial conditions using only a single boundary measurement.
Specifically, for two solutions $u_j~(j=1,2)$ of \eqref{eq:1-PR25} corresponding to different potentials $q_j$ and different initial conditions $a_j$, $b_j$, we establish a Lipschitz stability result of the following form,
\[
\norm{q_1 - q_2} + \norm{a_1 - a_2} + \norm{b_1 - b_2} \lesssim \norm{(u_1 - u_2)|_{\Sigma_T}} + \norm{\partial_\nu (u_1 - u_2)|_{\Sigma_T}}.
\]
To achieve this result, we impose a pointwise comparison estimate on the Laplacian of the differences of the initial conditions, together with a positivity condition on one of them.

\begin{asm} \label{asm:assu-PR25}
	Assume $a_j, b_j \in C^2(\overline \Omega)$ for $j=1,2$, and the following pointwise conditions hold in $\Omega$ for certain constant $\widetilde C > 0$,
	\begin{subequations} \label{eq:ass1-PR25}
		\begin{align}
			|\Delta (a_1 - a_2)| & \leq \widetilde C\big(|a_1 - a_2| + |\nabla (a_1 - a_2)|\big), \label{eq:as1a-PR25} \\
			|\Delta (b_1 - b_2)| & \leq \widetilde C\big(|b_1 - b_2| + |\nabla (b_1 - b_2)|\big). \label{eq:as1b-PR25}
		\end{align}
	\end{subequations}
\end{asm}

Assumption \ref{asm:assu-PR25} generalizes previous stability results, which typically require both $a_1 - a_2$ and $b_1 - b_2$ vanish.
We denote by $\partial_\nu u := \nu \cdot \nabla u$ the outer normal derivative defined on $\Sigma_T$.
Our first result is as the following.

\begin{thm} \label{thm:1-PR25}
	Let $u_j~(j=1,2)$ be weak solutions of \eqref{eq:1-PR25} with $q_j$, $a_j$, $b_j$, $g_j$ satisfying \eqref{eq:1con-PR25}, and the compatibility conditions up to order $s$.
	Let $M > 0$ be such that $\norm{q_j}_{C^{s+1}(\overline \Omega)} + \norm{a_j}_{H^{s+1}(\Omega)} + \norm{b_j}_{H^s(\Omega)} + \norm{g_j}_{H^{s+1/2}(\Sigma_T)} \leq M$ for both $j=1,2$.
	Assume there are two positive constants $C_1$, $C_2$ such that
	\begin{equation} \label{eq:ass2-PR25}
		\norm{q_1}_{L^\infty(\Omega)} \leq C_1 \ \ \mbox{and} \ \ \inf_{x \in \overline \Omega} a_2(x) \geq C_2.
	\end{equation}
	We further assume Assumption \ref{asm:assu-PR25} holds true, then there exists a constant $T_0 = T_0(\Omega) > 0$ such that for any $T \geq T_0$, one can find $C = C(C_1, C_2, \widetilde C, \Omega, T, s, M) > 0$ such that
	\begin{align}
		& \, \norm{q_1 - q_2}_{L^2(\Omega)} + \norm{a_1 - a_2}_{H^1(\Omega)} + \norm{b_1 - b_2}_{H^1(\Omega)} \nonumber \\
		\leq & \, C \big( \norm{u_1 - u_2}_{H^3(\Sigma_T)} + \norm{\partial_\nu (u_1 - u_2)}_{H^2(\Sigma_T)} \big). \label{eq:rs1-PR25}
	\end{align}
\end{thm}

\begin{rem}
	To the best of our knowledge, Theorem \ref{thm:1-PR25} provides the first result enabling the simultaneous stable recovery of the potential, initial position, and initial velocity using only a single lateral boundary measurement.
	When $a_1 = a_2$ and $b_1 = b_2$, Assumption \ref{asm:assu-PR25} is automatically satisfied, and thus Theorem \ref{thm:1-PR25} generalizes classical results such as those in \cites{BukhgeimKlibanov, Masahiro1999, IY01}.
	Furthermore, the proof does not use time reflection.
\end{rem}

\begin{rem}
	Both the Dirichlet and Neumann data appear on the right-hand-side of \eqref{eq:rs1-PR25}.
	This means both the boundary data can differ for $j=1$ and $j=2$.
	By contrast, in previous works such as \cites{IY01, Imanuvilov2001, Bellassoued2004log}, the boundary condition must be fixed.
	This demonstrates that Theorem \ref{thm:1-PR25} offers greater flexibility in the choice of boundary data.
\end{rem}

\begin{rem}
	It is worth noting that only one of $a_1$ and $a_2$ is required to satisfy the positivity condition.
	For instance, choose $a_2$ such that $\inf_\Omega a_2 = 2$ and $\sup_\Omega a_1 = 1$ and let $a_1 = a_2 - \frac 5 2$.
	In this case, $a_2$ satisfies \eqref{eq:age0-PR25} but $a_1$ does not.
	Nevertheless, since $a_1 - a_2$ satisfies \eqref{eq:ass1-PR25}, the conclusion of Theorem \ref{thm:1-PR25} still holds.
\end{rem}

\begin{rem}
	The result in Theorem \ref{thm:1-PR25} is closely related to the thermo- and photo-acoustic tomography (TAT/PAT) problem.
	Mathematically speaking, the TAT/PAT problem is concerned with recovering a sound speed $c(x)$ and an internal source $f(x)$ from the boundary measurement $u|_{(0,+\infty) \times \partial \Omega}$ of the following PDE,
	\begin{equation*}
		\left\{\begin{aligned}
			c(x)^{-2} \partial_t^2 u - \Delta u & = 0, && \text{in} \ (0,+\infty) \times \R^3, \\
			u(0,x) = f(x), \ \partial_t u(0,x) & = 0, && x \in \R^3.
		\end{aligned}\right.
	\end{equation*}
	Under certain monotonicity conditions, \cite{Kian2025pat} shows that the boundary data $u|_{(0,+\infty) \times \partial \Omega}$ uniquely determines $(c,f)$, see also \cite{KianLiu2025}.
	The simultaneous unique identifiability of $(f,c)$ was established in \cite{Liu2025determining} under a local energy decay condition and structural assumptions on the unknowns, without requiring monotonicity.
	Instead of using Carleman estimates, they generalize the method in \cite{Liu2015pat}, which expands the temporal Fourier transform of the solution with respect to the frequency.
	Compared to the TAT/PAT problem, the inverse problem of the IBVP considered in Theorem \ref{thm:1-PR25} enjoys Lipschitz stability.
	Moreover, this stability is achieved using only finite-time measurements, whereas the Fourier-based TAT/PAT approach generally requires measurements over an infinite time interval, as it relies on a temporal Fourier or Laplace transform.
\end{rem}

\begin{rem}
	The condition on the final time $T$ is given in Remark \ref{rem:T-PR25}.
\end{rem}

\smallskip

The inverse problem of simultaneously recovering the potential and initial data for \eqref{eq:1-PR25} is closely connected to the fixed angle inverse scattering problem for the following hyperbolic PDE (with $\theta \in \mathbb S^{n-1}$),
\begin{equation} \label{eq:1fix-PR25}
	\left\{\begin{aligned}
		(\square + q(x)) U(t,x) & = 0, && \text{in} \ \R^{1+n}, \\ 
		U(t,x) & = \delta(t - \theta \cdot x), && \text{when} \ t \ll 0.
	\end{aligned}\right.
\end{equation}
The fixed angle inverse scattering concerns the recovery of the potential $q$ from the total wave field $U|_{(-\infty,T) \times \partial \Omega}$, given that $\supp q \subset \Omega$.
To isolate the singularity associated with the $\delta$ source from the total wave $U$, one could use the progressive wave expansion
\(
U(t,x) = \delta(t - \theta \cdot x) + u(t,x) H(t - \theta \cdot x),
\)
then
\begin{equation} \label{eq:1fi2-PR25}
	\left\{\begin{aligned}
		(\square + q(x)) u(t,x) & = 0, && \text{in} \ \{t > \theta \cdot x \}, \\ 
		u(t,x) & = -\frac 1 2 \mathcal I q(x), && \text{on} \ \{t = \theta \cdot x \},
	\end{aligned}\right.
\end{equation}
where $\mathcal I$ is the X-ray transform in the direction $\theta$,
\[
\mathcal I q(x) := \int_{-\infty}^t q(x_0 + s\theta) \dif s, \quad \text{where} \quad x = x_0 + t\theta.
\]
The fixed angle inverse scattering problem using data generated by a single incident wave remains unresolved.
An important result in this regard is the work \cites{RakeshSalo2, RakeshSalo1} by Rakesh and Salo.
More recently, Oksanen, Rakesh, and Salo \cites{oksanen2024rigidityB, oksanen2024rigidityA} have established several rigidity results for the recovery of a sound speed in the fixed angle scattering problem.

Even though the scattering problem \eqref{eq:1fi2-PR25} and the IBVP \eqref{eq:1-PR25} are governed by the same hyperbolic equation, their corresponding inverse problems exhibit fundamentally different behaviors.
The most significant difference is the configuration of initial conditions.
They are different in the following two ways:
\begin{enumerate}
	\item[\textbf{(D1)}] The initial conditions for \eqref{eq:1-PR25} are prescribed on the horizontal hypersurface $\{t = 0\} \times \Omega$, whereas for the scattering problem \eqref{eq:1fi2-PR25}, the initial data is given on the characteristic surface $\{t = \theta \cdot x\}$ (a $45^\circ$ slanted plane in spacetime).
	
	\item[\textbf{(D2)}] In \eqref{eq:1-PR25}, the initial conditions are independent of the potential $q$, while in \eqref{eq:1fi2-PR25}, the initial condition is intrinsically coupled to $q$.
\end{enumerate}

These two differences pose challenges when applying stable recovery results such as Theorem \ref{thm:1-PR25} of \eqref{eq:1-PR25} to \eqref{eq:1fi2-PR25}.
\textbf{(D1)} implies that the Carleman estimates applicable to \eqref{eq:1-PR25} are no longer valid for \eqref{eq:1fi2-PR25} due to the characteristic initial surface.
To circumvent this, Rakesh and Salo used two incident waves in \cite{RakeshSalo1} to overcome this issue and obtained a Lipschitz stable recovery for the potential.
\textbf{(D2)} introduces a coupling between the initial data and the potential, which exacerbates the nonlinearity of the forward map $q \mapsto U|_{(-\infty, T) \times \partial \Omega}$.

Regarding these considerations, our second result focuses on the following IBVP,
\begin{equation} \label{eq:2-PR25}
	\left\{\begin{aligned}
		(\square + q) u & = 0 && \mbox{in} \ \ Q_T, \\
		u & = g && \mbox{on} \ \ \Sigma_T, \\
		u = q, \ \partial_t u & = 0 && \mbox{on} \ \ \{t = 0\} \times \Omega.
	\end{aligned}\right.
\end{equation}
We call it the hyperbolic \emph{initial-potential problem}.
We study the initial-potential problem because of the following reasons:
\begin{enumerate}
	\item The initial-potential problem \eqref{eq:2-PR25} mimics the scattering problem \eqref{eq:1fi2-PR25} in that the initial conditions in both are directly coupled to the potential.
	
	\item A key advantage of the initial-potential problem \eqref{eq:2-PR25} is that, as established in Theorem \ref{thm:2-PR25}, Lipschitz stability for the potential does \emph{not} require a positivity condition like \eqref{eq:age0-PR25}.
\end{enumerate}
We have the following result.

\begin{thm} \label{thm:2-PR25}
	Let $u_j~(j=1,2)$ be weak solutions of \eqref{eq:2-PR25} with $g_j \in H^{s+1/2}(\Sigma_T)$ and $q_j \in C^{s+1}(\overline \Omega)$ with $s \geq \max\{\lceil n/2 \rceil + 1, 3\}$, and the compatibility conditions up to order $s$.
	Let $M > 0$ be such that $\norm{g_j}_{H^{s+1/2}(\Sigma_T)} + \norm{q_j}_{C^{s+1}(\overline \Omega)} \leq M$.
	Then, there exists a constant $T_0 = T_0(\Omega) > 0$ such that for any $T \geq T_0$, one can find $C = C(\Omega, T, s, M) > 0$ such that
	\begin{equation*}
		C \norm{q_1 - q_2}_{H^1(\Omega)}
		\leq \norm{u_1 - u_2}_{H^1((0,T)\times \partial \Omega)} + \norm{\partial_\nu (u_1 - u_2)}_{L^2((0,T)\times \partial \Omega)}.
	\end{equation*}
\end{thm}

Besides the Lipschitz stability of the initial-potential problem, we also propose the following iterative reconstruction scheme.

\begin{thm} \label{thm:3-PR25}
	Let $u$ be weak solution of \eqref{eq:2-PR25} with $g \in H^{s+1/2}(\Sigma_T)$ and $q \in C^{s+1}(\overline \Omega)$ with $s \geq \max\{\lceil n/2 \rceil + 1, 3\}$, and the compatibility conditions up to order $s$.
	Let $M > 0$ be such that $\norm{g}_{H^{s+1/2}(\Sigma_T)} + \norm{q}_{C^{s+1}(\overline \Omega)} \leq M$.
	Using the Cauchy data $(g, \partial_\nu u)|_{\Sigma_T}$, there exists $T_0 = T_0(\Omega) > 0$ such that for all $T \geq T_0$, one can construct $\{ q_i^{(j)} \}_{i,j \geq 1}$ with $q_i^{(j)} := T_M \big( u_i^{(j)}(0,\cdot) \big)$, where $T_M$ is the truncation defined by \eqref{eq:TM-PR25} and $u_i^{(j)}$'s are given in \eqref{eq:pg-PR25}, such that
	\begin{equation} \label{eq:qij-PR25}
		\lim_{j \to \infty} q_i^{(j)} = q_i \ \mbox{in} \ L^2(\Omega),
		\qquad \mbox{and} \qquad
		\lim_{i \to \infty} q_i = q \ \mbox{in} \ L^2(\Omega).
	\end{equation}
\end{thm}

We briefly describe the scheme here; the precise definitions of $\mathcal A$, $\mathcal J_i$ and $T_M$ are given in Section \ref{sec:schm-PR25}.
Starting from $q_0=0$, repeat the following steps:
\begin{enumerate}
	\item Freeze the coefficient at $q_i$, and minimize $\mathcal J_i$ over $\mathcal A$. Let $u_i$ be the minimizer.
	
	\item Update $q_{i+1}=T_M(u_i(0,\cdot))$.
	
	\item Increase $i$ by $1$, repeat (1) and (2).
\end{enumerate}
By Theorem \ref{thm:3-PR25}, $q_i \to q$ in $L^2(\Omega)$.
The inner solver and the exact form of $\mathcal J_i$ are described in Section \ref{sec:schm-PR25}.

Compared with the convexification method \cites{Klibanov1997global, Kl13, KlibanovLiZhang2019, Klibanov2021book}, the scheme in Section \ref{sec:schm-PR25} freezes the coefficient at each step and optimizes the state $u$.
The inner problem is strictly convex on the admissible set, and the coefficient update $q_{i+1} = T_M(u_i(0,\cdot))$ is explicit.
In contrast to the Tikhonov-regularized formulations of convexification, no Tikhonov penalty is present, and for exact data the limit is the true potential $q$, which is the unique fixed point in the class used in Proposition \ref{prop:cntr-PR25}.
	
The closest algorithm of this type is due to Baudouin, de Buhan and Ervedoza \cite{Baudouin2017conv}.
They use a Carleman-weighted least-squares functional and a trace update for the standard potential problem, but their update is a linearized correction that divides by the known initial displacement $w_0$, so it requires $|w_0| \geq \alpha > 0$.
In the initial-potential problem, the initial displacement is the unknown $q$ itself; the update becomes the fixed-point map $q_{i+1} = T_M(u_i(0,\cdot))$, and no positivity condition is needed.

Finally, we present a new proof of a pointwise Carleman estimate in Proposition \ref{prop:pwCa-PR25}.
While such estimates exist in the literature—for example, in \cite{Shishatskij1986}*{\S4.4 Lemma 1}, \cites{Kubo2000Unique}, \cite{LvQi2019Unified}*{Corollary 4.1}, \cite{RakeshSalo2}*{Appendix A}, and \cite{Yamamoto2025book}—our result generalizes the one in \cite{Shishatskij1986}*{\S4.4 Lemma 1} to the case of an anisotropic leading order term.
Furthermore, by leveraging an algebraic inequality (Lemma \ref{lem:ajg-PR25}), our proof is significantly shorter, with clearer reasoning and computations compared to existing proofs, such as those in \cite{Shishatskij1986}*{\S4.4 Lemma 1} and \cite{LvQi2019Unified}*{Corollary 4.1}. The explicit boundary terms in Proposition \ref{prop:pwCa-PR25} play a crucial role in proving Theorems \ref{thm:1-PR25} and \ref{thm:2-PR25}.

The rest of this article is organized as follows.
In Section \ref{sec:2est-PR25} we prove a weighted energy estimate and a Carleman estimate which are necessary for our analysis.
In Section \ref{subsec:Carl-PR25} we first establish a new pointwise Carleman estimate with boundary terms, then we adopt the pointwise Carleman estimate to our case and obtain an integral form of it.
The proof of Theorem \ref{thm:1-PR25} is presented in Section \ref{sec:IBVP-PR25}, and the proof of Theorem \ref{thm:2-PR25} and the reconstruction scheme is given in Section \ref{sec:schm-PR25}.

\section{Preliminaries} \label{sec:2est-PR25}

In this section we first record an $L^\infty$ bound for solutions of \eqref{eq:1-PR25} and \eqref{eq:2-PR25}, and then prove a weighted
energy estimate and a Carleman estimate.
These are combined to establish Proposition \ref{prop:enCr-PR25}, which provides an estimate of the initial conditions
using only the lateral boundary and source data.

\subsection{\texorpdfstring{$L^\infty$}{L-infinity} bound}

We begin with the following bound, which will be used in the absorption arguments below.

\begin{lem} \label{lem:Linf-PR25}
	Let $m = \lfloor n/2\rfloor+1$.
	Suppose $q\in W^{m-1,\infty}(\Omega)$, $g\in H^{m+1/2}(\Sigma_T)$, $a\in H^{m+1/2}(\Omega)$, $b\in H^{m-1/2}(\Omega)$, and the compatibility conditions hold up to order $m$.
	Then the solution of \eqref{eq:1-PR25} satisfies $u\in C(\overline Q_T)$ and
	\begin{equation*}
		\norm{u}_{L^\infty(Q_T)}
		\leq C(T,\Omega,m)\,
		e^{C (1 + \norm{q}_{W^{m-1,\infty}}) T}
		( \norm{g}_{H^{m+1/2}(\Sigma_T)} + \norm{a}_{H^{m+1/2}(\Omega)} + \norm{b}_{H^{m-1/2}(\Omega)} ).
	\end{equation*}
\end{lem}

\begin{proof}
	Choose $G\in H^{m+1}(Q_T)$ such that $G|_{\Sigma_T}=g$, $G(0,\cdot)=a$, $\partial_tG(0,\cdot)=b$, and
	\[
	\norm{G}_{H^{m+1}(Q_T)}
	\leq C\big(\norm{g}_{H^{m+1/2}(\Sigma_T)}
	+\norm{a}_{H^{m+1/2}(\Omega)}
	+\norm{b}_{H^{m-1/2}(\Omega)}\big).
	\]
	Then $w:=u-G$ satisfies $(\square+q)w=-(\square+q)G$ with zero boundary and zero initial data.
	The $H^m$ energy estimate gives
	\[
	\norm{w}_{L^\infty(0,T;H^m)}
	+\norm{\partial_tw}_{L^\infty(0,T;H^{m-1})}
	\leq C e^{C(1+\norm{q}_{W^{m-1,\infty}})T}
	\norm{(\square+q)G}_{L^2(0,T;H^{m-1})}.
	\]
	Since $\norm{(\square+q)G}_{L^2H^{m-1}} \leq C(1+\norm{q}_{W^{m-1,\infty}})\norm{G}_{H^{m+1}}$, $m>n/2$ implies
	$H^m(\Omega) \hookrightarrow C(\overline\Omega)$, and $m+1>(n+1)/2$ implies $G\in H^{m+1}(Q_T) \hookrightarrow
	C(\overline Q_T)$.
	Combining these estimates proves the result.
\end{proof}

Corollary \ref{cor:Linf-PR25} follows directly from Lemma \ref{lem:Linf-PR25}.

\begin{cor} \label{cor:Linf-PR25}
	Let $m = \lfloor n/2\rfloor+1$. Suppose
	$q\in W^{m-1,\infty}(\Omega)\cap H^{m+1/2}(\Omega)$ and $g\in H^{m+1/2}(\Sigma_T)$, and the compatibility conditions hold up to order $m$.
	Then the solution of \eqref{eq:2-PR25} satisfies $u\in C(\overline Q_T)$ and
	\begin{equation*}
		\norm{u}_{L^\infty(Q_T)}
		\leq C(T,\Omega,m)\,
		e^{C(1+\norm{q}_{W^{m-1,\infty}})T}
		\big(\norm{g}_{H^{m+1/2}(\Sigma_T)}
		+\norm{q}_{H^{m+1/2}(\Omega)}\big).
	\end{equation*}
\end{cor}

\subsection{Weighted energy estimate} \label{subsec:WEE-PR25}

For any function $u$, we adopt the following convention:
\begin{equation*} %\label{eq:habv-PR25}
	u_t := \partial_t u, \quad
	u_j := \partial_{x_j} u, \quad
	\nabla u := (\partial_{x_1} u, \cdots, \partial_{x_n} u), \quad
	\nabla_{t,x} u := (\nabla_t u, \nabla u).
\end{equation*}
In what follows, for simplicity we may omit ``$\dif x$'' (resp.~``$\dif t \dif x$'') in any integral $\int [\cdots] \dif x$ (resp.~$\int [\cdots] \dif t \dif x$) unless otherwise stated.
We also use $C$ and its variants, such as $C_1$, to denote some generic constants whose particular values may change line by line.
For two quantities, we write $P \lesssim Q$ (resp.~$P \gtrsim Q$) to signify $P \leq CQ$ (resp.~$P \geq CQ$), for some generic positive constant $C$.

In what follows, denote
\[
\Omega_t := \{ (t,x) \in \R^{1+n} \,;\, x \in \Omega \} \ \ \mbox{for} \ \ t \geq 0,
\]
and we identify $\Omega_0$ with $\Omega$.
For a test function $p$ we have
\begin{equation*}
	\left\{\begin{aligned}
		\square (pu) & = (p \square + 2\mathcal L_p + \square p) u, \\
		\mathcal L_{e^{\lambda \varphi}} & = \lambda e^{\lambda \varphi} \mathcal L_\varphi, \\
		\square e^{\lambda \varphi} & = \lambda^2 e^{\lambda \varphi} (\mathcal L_\varphi \varphi + \lambda^{-1} \square \varphi),
	\end{aligned}\right.
\end{equation*}
where
\begin{equation*}
	\mathcal L_p := \partial_t p \partial_t - \nabla p \cdot \nabla.
\end{equation*}
Note that $\mathcal L_p$ is a first-order differential operator.
Hence, for $v = e^{\lambda \varphi} u$ we have
\begin{align}
	\square v
	& = e^{\lambda \varphi} \big[ \square + 2\lambda \mathcal L_\varphi + \lambda^2 (\mathcal L_\varphi \varphi + \lambda^{-1} \square \varphi) \big] u \label{eq:spe1-PR25} \\
	& = e^{\lambda \varphi} \square u + 2\lambda \mathcal L_\varphi v + \lambda^2 (-\mathcal L_\varphi \varphi + \lambda^{-1} \square \varphi) v. \label{eq:ape2-PR25}
\end{align}
where we used $\mathcal L_\varphi (e^{-\lambda \varphi} v) = e^{-\lambda \varphi} \mathcal L_\varphi v + (\mathcal L_\varphi e^{-\lambda \varphi}) v$ and $\mathcal L_\varphi e^{-\lambda \varphi} = -\lambda e^{-\lambda \varphi} \mathcal L_\varphi \varphi$ due to the first-order property of $\mathcal L_\varphi$.

\begin{prop}[Weighted energy estimate] \label{prop:weg1-PR25}
	Assume $\varphi \in C^2(\overline Q_T)$.
	There exists a universal constant $C > 0$, and $\lambda_0 > 0$ such that for all $\lambda \geq \lambda_0 + 1$,
	\begin{align}
		C\int_{\Omega_0} e^{2\lambda \varphi} (|\nabla_{t,x} u|^2 + \lambda^2 u^2)
		& \leq \int_{Q_T} e^{2\lambda \varphi} |\square u|^2 + \int_{\Sigma_T} |\partial_t (e^{\lambda \varphi} u) \partial_\nu(e^{\lambda \varphi} u)| \nonumber \\
		& \quad + \mathcal C_\varphi(T) \int_{Q_T} e^{2\lambda \varphi} (\lambda |\nabla_{t,x} u|^2 + \mathcal C_\varphi(T) \lambda^3 u^2) \nonumber \\
		& \quad + \int_{\Omega_T} e^{2\lambda \varphi} (|\nabla_{t,x} u|^2 + \mathcal C_\varphi(T) \lambda^2 u^2), \label{eq:h0u-PR25}
	\end{align}
	where
	\begin{equation*}
		\mathcal C_\varphi(T) := \sup_{Q_T} (|\nabla_{t,x} \varphi| + 1)^2 + \sup_{Q_T} |\square \varphi| + 1.
	\end{equation*}
\end{prop}

\begin{proof}
	Set $v = e^{\lambda \varphi} u$ and denote
	\[
	e(t) := \frac 1 2 \int_{\Omega_t} (|\nabla_{t,x} v|^2 + c_0 \lambda^2 v^2) \dif x.
	\]
	Add a constant $c_0$ to \eqref{eq:ape2-PR25} to counter the possibly negative effect of $-\mathcal L_\varphi \varphi$,
	\begin{equation*}
		(\square + c_0 \lambda^2) v
		= e^{\lambda \varphi} \square u + 2\lambda \mathcal L_\varphi v + \lambda^2 (c_0 - \mathcal L_\varphi \varphi + \lambda^{-1} \square \varphi) v.
	\end{equation*}
	Multiplying both sides by $v_t$ gives
	\begin{align}
		& \divr_{t,x} \big( \frac 1 2 (|\nabla_{t,x} v|^2 + c_0 \lambda^2 v^2), -v_t \nabla v \big) \nonumber \\
		= & \ e^{\lambda \varphi} (\square u) v_t + 2\lambda (\mathcal L_\varphi v) v_t + \lambda^2 (c_0 - \mathcal L_\varphi \varphi + \lambda^{-1} \square \varphi) v v_t. \label{eq:aL3-PR25}
	\end{align}
	Integrating \eqref{eq:aL3-PR25}, and substituting $v = e^{\lambda \varphi} u$, we have
	\begin{align}
		e(0)
		& = e(T) - \int_{\Sigma_T} v_t v_\nu - \int_{Q_T} [e^{\lambda \varphi} (\square u) v_t + 2\lambda (\mathcal L_\varphi v) v_t + \lambda^2 (c_0 - \mathcal L_\varphi \varphi + \lambda^{-1} \square \varphi) v v_t] \nonumber \\
		& \leq e(T) + \int_{\Sigma_T} |v_t v_\nu| + C\int_{Q_T} \big[ e^{2\lambda \varphi} |\square u|^2 + (\mathcal C_\varphi(T) + c_0) (\lambda |\nabla_{t,x} v|^2 + \lambda^3 v^2) \big]. \label{eq:aL6-PR25}
	\end{align}
	By the definition of $e(t)$, we have
	\begin{align*}
		e(0)
		& \geq \frac 1 2 \int_{\Omega_0} e^{2\lambda \varphi} \big[ \frac 1 2 |\nabla_{t,x} u|^2 + \lambda^2 (c_0 - |\nabla_{t,x} \varphi|^2) u^2 \big].
	\end{align*}
	Taking
	\(
	c_0 = \sup_{\Omega_0} |\nabla_{t,x} \varphi|^2 + 1,
	\)
	from \eqref{eq:aL6-PR25} we obtain
	\begin{align*}
		& C \int_{\Omega_0} e^{2\lambda \varphi} (|\nabla_{t,x} u|^2 + \lambda^2 u^2) \nonumber \\
		\leq & \, e(0) \leq e(T) + \int_{\Sigma_T} |v_t v_\nu| + C \mathcal C_\varphi(T) \int_{Q_T} \big[ e^{2\lambda \varphi} |\square u|^2 + (\lambda |\nabla_{t,x} v|^2 + \lambda^3 v^2) \big], %\label{eq:aL4-PR25}
	\end{align*}
	and
	\[
	e(T) \leq C \int_{\Omega_T} e^{2\lambda \varphi} (|\nabla_{t,x} u|^2 + \mathcal C_\varphi(T) \lambda^2 u^2),
	\]
	for certain universal constant $C$.
	Combining these two gives \eqref{eq:h0u-PR25}.
	The proof is done.
\end{proof}

\subsection{Carleman estimate with boundary terms} \label{subsec:Carl-PR25}

We shall use the following lemma in the proof of a pointwise Carleman estimate.

\begin{lem} \label{lem:ajg-PR25}
	For any real numbers $a_1 $, $\cdots$, $a_n$ and $b$, and any integers $j_0, k_0 \in \{1,\cdots,n\}$, there holds
	\[
	\big( \sum_{j=1}^n a_j \big)^2
	\geq 2 \sum_{1 \leq j < k \leq n} a_j a_k + 2(a_{j_0} - a_{k_0}) b - 2b^2.
	\]
\end{lem}

\begin{proof}
	When $j_0 = k_0$, it is trivial.
	Otherwise, changing $a_{j_0}$ to $a_{j_0} - b$ and $a_{k_0}$ to $a_{k_0} + b$ gives the result.
\end{proof}

A common way to prove Carleman estimates is to use G\"arding's inequality and positivity of the commutator.
But these methods usually cannot yield boundary terms in explicit form.
When the function is not compactly supported, direct computations are more feasible to obtain explicit expressions of boundary terms.
The following claim shows a pointwise Carleman estimate with explicit boundary terms.
We adopt the Einstein summation convention for repeated indices.

\begin{prop}[Pointwise Carleman estimate] \label{prop:pwCa-PR25}
	Let $a^{jk}(x)$ be a real symmetric matrix-valued, $C^2$-smooth function.
	Let $u$ and $\tilde \ell$ be real-valued and $C^2$-smooth, and $\ell$ be real-valued and $C^3$-smooth.
	Then for $v = e^\ell u$ there holds,
	\begin{align}
		e^{2\ell} |a^{jk} \partial_{jk} u|^2
		& \geq 2 \big[ 2 (a^{jk'} a^{j'k} \ell_{j'})_{k'} - (a^{jk} a^{j'k'})_{k'} \ell_{j'} - \tilde \ell a^{jk} \big] v_j v_k + 2 (a^{jk} a^{j'k'} \ell_{j'k'} - a^{jk} \tilde \ell)_j v v_k \nonumber \\
		& + \big\{ 2 a^{jk} a^{j'k'} (\ell_{j'} \ell_{k'})_j \ell_k + 2(a^{jk} a^{j'k'})_j \ell_{j'} \ell_{k'} \ell_k + 2a^{jk} \ell_j \ell_k \tilde \ell \big\} v^2 \nonumber \\
		& - \big\{ 2 a^{jk} a^{j'k'} (\ell_k \ell_{j'k'})_j + 2(a^{jk} a^{j'k'})_j \ell_{j'k'} \ell_k + 2\tilde \ell^2 - 2a^{jk} \ell_{jk} \tilde \ell \big\} v^2 + \partial_j F^j, \label{eq:ePu2-PR25}
	\end{align}
	where
	\begin{equation} \label{eq:eP2F-PR25}
		F^j := a^{jk} a^{j'k'} (2 \ell_k v_{j'} v_{k'} - 4 \ell_{j'} v_{k} v_{k'} - 2\ell_{j'k'} v_k v - 2 \ell_{j'} \ell_{k'} \ell_k v^2 + 2\ell_k \ell_{j'k'} v^2) + 2\tilde \ell a^{jk} v_k v.
	\end{equation}
\end{prop}

\begin{proof}
	We have
	\begin{equation*}
		e^{\ell} a^{jk} u_{jk}
		= e^{\ell} a^{jk} (e^{-\ell} v)_{jk}
		= a^{jk} (v_{jk} + \ell_j \ell_k v) - 2 a^{jk} \ell_j v_k - a^{jk} \ell_{jk} v
		=: I_2 + I_1 + I_0,
	\end{equation*}
	so
	\begin{equation} \label{eq:I123-PR25}
		e^{2\ell} |a^{jk}(x) u_{jk}|^2
		\geq 2I_2 I_1 + 2I_2 I_0 + 2I_1 I_0.
	\end{equation}
	We analyze the terms in \eqref{eq:I123-PR25} on-by-one.
	For $2 I_2 I_1$,
	\begin{equation*}
		2 I_2 I_1
		= -4 (a^{jk} v_{jk} + a^{jk} \ell_j \ell_k v) (a^{j'k'} \ell_{j'} v_{k'})
		=: -4(I_{21}^1 + I_{21}^2).
	\end{equation*}
	We have
	\begin{align*}
		I_{21}^1
		& = a^{jk} a^{j'k'} \ell_{j'} v_{jk} v_{k'}
		= \big( a^{jk} a^{j'k'} \ell_{j'} v_{k} v_{k'} \big)_j - \big( a^{jk} a^{j'k'} \ell_{j'} v_{k'} \big)_j v_k \\
		& = \big( a^{jk} a^{j'k'} \ell_{j'} v_{k} v_{k'} \big)_j - \big( a^{jk} a^{j'k'} \ell_{j'} \big)_j v_{k'} v_k - a^{jk} a^{j'k'} \ell_{j'} v_{jk'} v_k.
	\end{align*}
	The symmetry of $a^{jk}$ guarantees that the last term satisfies
	\begin{align*}
		a^{jk} a^{j'k'} \ell_{j'} v_{jk'} v_k
		& = \frac 1 2 a^{jk} a^{j'k'} \ell_{j'} (v_{jk'} v_k + v_{kk'} v_j)
		= \frac 1 2 a^{jk} a^{j'k'} \ell_{j'} \big( v_j v_k \big)_{k'} \\
		& = \frac 1 2 \big( a^{jk} a^{j'k'} \ell_{j'} v_j v_k \big)_{k'} - \frac 1 2 \big( a^{jk} a^{j'k'} \ell_{j'} \big)_{k'} v_j v_k,
	\end{align*}
	so
	\begin{equation*}
		I_{21}^1
		= \big( a^{jk} a^{j'k'} (\ell_{j'} v_{k} v_{k'} - \frac 1 2 \ell_k v_{j'} v_{k'}) \big)_j - \big( a^{jk} a^{j'k'} \ell_{j'} \big)_j v_{k'} v_k + \frac 1 2 \big( a^{jk} a^{j'k'} \ell_{j'} \big)_{k'} v_j v_k.
	\end{equation*}
	Moreover,
	\begin{align*}
		I_{21}^2
		& = (a^{jk} \ell_j \ell_k v) (a^{j'k'} \ell_{j'} v_{k'})
		= \frac 1 2 a^{jk} a^{j'k'} \ell_j \ell_k \ell_{j'} \big( v^2 \big)_{k'}
		= \frac 1 2 a^{j'k'} a^{jk} \ell_{j'} \ell_{k'} \ell_k \big( v^2 \big)_j \\
		& = \frac 1 2 \big( a^{jk} a^{j'k'} \ell_{j'} \ell_{k'} \ell_k v^2 \big)_j - \frac 1 2 \big( a^{jk} a^{j'k'} \ell_{j'} \ell_{k'} \ell_k \big)_j v^2.
	\end{align*}
	Substituting $I_{21}^1$ and $I_{21}^2$ into $2I_2 I_1$ gives
	\begin{align}
		2 I_2 I_1
		& = 4 \big( a^{jk} a^{j'k'} \ell_{j'} \big)_j v_{k'} v_k - 2 \big( a^{jk} a^{j'k'} \ell_{j'} \big)_{k'} v_j v_k + 2 \big( a^{jk} a^{j'k'} \ell_{j'} \ell_{k'} \ell_k \big)_j v^2 \nonumber \\
		& \quad + \big( a^{jk} a^{j'k'} (2 \ell_k v_{j'} v_{k'} - 4 \ell_{j'} v_{k} v_{k'} - 2 \ell_{j'} \ell_{k'} \ell_k v^2) \big)_j. \label{eq:I12-PR25}
	\end{align}
	For $2 I_2 I_0$,
	\begin{align}
		2 I_2 I_0
		& = -2 a^{jk} (v_{jk} + \ell_j \ell_k v) (a^{j'k'} \ell_{j'k'} v)
		= -2 a^{jk} a^{j'k'} \ell_j \ell_k \ell_{j'k'} v^2 - 2 a^{jk} a^{j'k'} \ell_{j'k'} v_{jk} v \nonumber \\
		& = -2 a^{jk} a^{j'k'} \ell_j \ell_k \ell_{j'k'} v^2 + 2 \big( a^{jk} a^{j'k'} \ell_{j'k'} \big)_j v v_k + 2 a^{jk} a^{j'k'} \ell_{j'k'} v_j v_k \nonumber \\
		& \quad - 2 \big( a^{jk} a^{j'k'} \ell_{j'k'} v_k v \big)_j. \label{eq:I13-PR25}
	\end{align}
	For $2 I_1 I_0$,
	\begin{equation}
		2 I_1 I_0
		= 2 a^{jk} a^{j'k'} \ell_k \ell_{j'k'} \big( v^2 \big)_j
		= - 2 \big( a^{jk} a^{j'k'} \ell_k \ell_{j'k'} \big)_j v^2 + 2 \big( a^{jk} a^{j'k'} \ell_k \ell_{j'k'} v^2 \big)_j. \label{eq:I23-PR25}
	\end{equation}
	
	Later in \eqref{eq:ePv3-PR25} we shall see that the coefficient of $v_t^2$ will be negative if we do not utilize $\tilde \ell$.
	To compensate for this negativity, we \textit{borrow} some positivity from the spatial derivatives through $\tilde \ell$.
	Specifically, we use Lemma \ref{lem:ajg-PR25} to add two more terms to \eqref{eq:I123-PR25},
	\begin{equation} \label{eq:II13-PR25}
		e^{2\ell} |a^{jk} u_{jk}|^2
		\geq 2I_2 I_1 + 2I_2 I_0 + 2I_1 I_0 + 2(I_2 - I_0) \tilde \ell v - 2\tilde \ell^2 v^2.
	\end{equation}
	For $2I_2 \tilde \ell v$, direct computation gives
	\begin{equation} \label{eq:I1l-PR25}
		2I_2 \tilde \ell v
		= \big( 2\tilde \ell a^{jk} v_k v \big)_j - 2\tilde \ell a^{jk} v_j v_k - 2 \big( a^{jk} \tilde \ell \big)_j v v_k + 2a^{jk} \ell_j \ell_k \tilde \ell v^2.
	\end{equation}
	It is for the term $-2\tilde \ell a^{jk} v_j v_k$ that we need in order to introduce $\tilde \ell v$.
	Combining \eqref{eq:II13-PR25}, \eqref{eq:I12-PR25}, \eqref{eq:I13-PR25}, \eqref{eq:I23-PR25} and \eqref{eq:I1l-PR25}, we arrive at the result.
\end{proof}

\begin{rem}
	In contrast to the standard approach of decomposing the conjugated operator $e^{\ell} a^{jk} \partial_{jk} e^{-\ell}$ into self-adjoint and anti-self-adjoint parts, our proof proceeds by a direct computation of key terms.
	The central idea is to introduce an auxiliary function $\tilde \ell$, which allows us to effectively transfer part of the positivity from the spatial derivatives to compensate for the negativity arising from the time derivative.
\end{rem}

By suitably choosing $a^{jk}$, $\varphi$ and $\tilde \varphi$, and integrating the inequality, from Proposition \ref{prop:pwCa-PR25} we obtain the following Carleman estimate for the wave operator.

\begin{prop}[Carleman estimate] \label{prop:InCa-PR25}
	Set the phase function $\varphi(t,x)$ as
	\[
	\varphi(t,x) := \frac 1 2 |x - \eta|^2 - \frac \beta 2 t^2,
	\]
	where $\eta \in \Rn \backslash \overline \Omega$ is chosen such that $\inf_\Omega |\nabla \varphi| \geq 1$.
	For any $u \in H^2(Q_T)$ and any $(\beta,T,\lambda)$ determined by the following restrictions:
	\begin{equation} \label{eq:bTl-PR25}
		\left\{\begin{aligned}
			T_0 & := 40 (\sup_\Omega |\nabla \varphi| + 1)^3, &
			T & \geq T_0, &
			\beta & := 2(\sup_\Omega |\nabla \varphi| + 1) T^{-1}, \\
			\lambda_0 & := \beta^{-1} + 1, &
			\lambda & \geq \lambda_0, &
			C_2 & := C_2(T_0, \lambda_0) > 0,
		\end{aligned}\right.
	\end{equation}
	there holds
	\begin{align}
		& C_2 \Big( \int_{Q_T} e^{2\lambda \varphi} (\beta \lambda |\nabla_{t,x} u|^2 + \lambda^3 u^2) + \int_{\Omega_T} e^{2\lambda \varphi} (\lambda |\nabla_{t,x} u|^2 + \lambda^3 u^2) + \int_{\Omega_0} e^{2\lambda \varphi} \mathcal B \Big) \nonumber \\
		\leq & \int_{Q_T} e^{2\lambda \varphi} |\square u|^2 + \int_{\Sigma_T} e^{2\lambda \varphi} (\lambda |\nabla_{t,x} u|^2 + \lambda^3 u^2), \label{eq:InCa-PR25}
	\end{align}
	with
	\[
	\mathcal B
	: = u_t [{-}4\lambda^2 |\nabla \varphi|^2 u - 4\lambda \nabla \varphi \cdot \nabla u - (2n - 4\beta) \lambda u]|_{t=0}.
	\]
\end{prop}

\begin{proof}
	Since $C^2(\overline Q_T)$ is dense in $H^2(Q_T)$ and all terms in \eqref{eq:InCa-PR25} are continuous under $H^2(Q_T)$ convergence, it suffices to prove the estimate for $u \in C^2(\overline Q_T)$.
	We therefore assume $u \in C^2(\overline Q_T)$, so that Proposition \ref{prop:pwCa-PR25} applies.
	We regard the variable $x$ in Proposition \ref{prop:pwCa-PR25} as $(t,x) \in \R^{n+1}$ where $t = x_0 \in \mathbb R$ and $x = (x_1,\cdots, x_n) \in \Rn$.
	Let $a^{jk}$, $\ell$ and $\tilde \ell$ be
	\[
	a^{jk}(t,x) =
	\begin{cases}
		-1, & j = k = 0, \\
		1, & 1 \leq j = k \leq n, \\
		0, & j \neq k,
	\end{cases}
	\quad
	\ell(t,x) := \lambda \varphi(t,x), \quad
	\tilde \ell(t,x) := \lambda \tilde \varphi,
	\]
	so that $a^{jk} \partial_{jk} = \Delta_x - \partial_t^2 = -\square$, and \eqref{eq:ePu2-PR25} and \eqref{eq:eP2F-PR25} can be simplified as
	\begin{align}
		e^{2\ell} |\square u|^2
		& \geq 4 \lambda \big( \sum_{j,k=1}^n \varphi_{jk} v_j v_k - 2 \sum_{j=1}^n \varphi_{tj} v_t v_j + \varphi_{tt} v_t^2 \big) + 2\lambda \tilde \varphi (v_t^2 - |\nabla v|^2) \nonumber \\
		& \quad + \lambda^3 \big[ (2 \nabla \varphi \cdot \nabla - 2\varphi_t \partial_t + 2\tilde \varphi) (|\nabla \varphi|^2 - \varphi_t^2) \big] v^2 + \mathcal O(|\nabla_{t,x} v|^2 + \lambda^2 v^2) \nonumber \\
		& \quad + \partial_t F^0 + \sum_{j=1}^n \partial_j F^j, \label{eq:eP1-PR25}
	\end{align}
	where $v = e^{\lambda \varphi} u$ and
	\begin{equation} \label{eq:eP2-PR25}
		\left\{\begin{aligned}
			F^0 & = -2\lambda \varphi_t (|\nabla v|^2 - v_t^2) + 4\lambda (\nabla \varphi \cdot \nabla v - \varphi_t v_t) v_t - 2\lambda (\square \varphi + \tilde \varphi) v_t v \\
			& \quad + 2\lambda^3 \varphi_t (|\nabla \varphi|^2 - \varphi_t^2) v^2 + 2\lambda^2 \varphi_t (\square \varphi) v^2, \\
			%%%%%
			F^j & = 2 \lambda \varphi_j (|\nabla v|^2 - v_t^2) - 4 \lambda (\nabla \varphi \cdot \nabla v - \varphi_t v_t) v_j + 2\lambda (\square \varphi + \tilde \varphi) v_j v \\
			& \quad - 2\lambda^3 \varphi_j (|\nabla \varphi|^2 - \varphi_t^2) v^2 - 2\lambda^2 \varphi_j (\square \varphi) v^2. \qquad (1 \leq j \leq n)
		\end{aligned}\right.
	\end{equation}
	Denote $\phi(x) := \frac 1 2 |x - \eta|^2$, $\tau(t) = -\frac \beta 2 t^2$, then $\varphi(t,x) = \phi(x) + \tau(t)$.
	After specifying $\varphi$, \eqref{eq:ePu2-PR25} and \eqref{eq:eP2F-PR25} can be further simplified as
	\begin{align}
		e^{2\lambda \varphi} |\square u|^2
		\geq & \ (4 - 2\tilde \varphi) \lambda |\nabla v|^2 + (2\tilde \varphi - 4\beta) \lambda v_t^2 + \mathcal O(|\nabla_{t,x} v|^2 + \lambda^2 v^2) \nonumber \\
		& + [(4 + 2\tilde \varphi) |\nabla \phi|^2 - (4\beta + 2\tilde \varphi) \dot \tau^2] \lambda^3 v^2 + \partial_t F^0 + \sum_{j=1}^n \partial_j F^j, \label{eq:ePv3-PR25}
	\end{align}
	and
	\begin{equation} \label{eq:e22F-PR25}
		\left\{\begin{aligned}
			F^0 & = -2\lambda \dot \tau (|\nabla v|^2 - v_t^2) + 4\lambda (\nabla \phi \cdot \nabla v - \dot \tau v_t) v_t + 2\lambda^3 \dot \tau (|\nabla \phi|^2 - \dot \tau^2) v^2 \\
			& \quad - 2\lambda (\square \varphi + \tilde \varphi) v_t v + 2\lambda^2 \varphi_t (\square \varphi) v^2, \\
			& \geq (-\dot \tau - |\nabla \phi|) \big[ 2\lambda |\nabla_{t,x} v|^2 + 2\lambda^3 (-\dot \tau) (-\dot \tau + |\nabla \phi|) v^2 \big] + \mathcal O(|\nabla_{t,x} v|^2 + \lambda^2 v^2), \\
			%%%%%
			F^j & = \mathcal O(\lambda |\nabla_{t,x} v|^2 + \lambda^3 v^2), \qquad 1 \leq j \leq n.
		\end{aligned}\right.
	\end{equation}
	
	Next, we take advantage of \eqref{eq:bTl-PR25} to further simplify \eqref{eq:ePv3-PR25} and \eqref{eq:e22F-PR25}.
	Set $\tilde \varphi = 3\beta$.
	It is easy to see $\beta < 1/3$, so $4 - 2\tilde \varphi > 2$.
	For $(t,x) \in Q_T$, the coefficient in front of $\lambda^3 v^2$ in \eqref{eq:ePv3-PR25} satisfies
	\begin{align*}
		[(4 + 2\tilde \varphi) |\nabla \phi|^2 - (4\beta + 2\tilde \varphi) \dot \tau^2]
		& \geq 4|\nabla \phi|^2 - 10\beta^3 t^2
		\geq 4 - 10\beta^3 T^2 \\
		& = 4 - 80(\sup_\Omega |\nabla \varphi| + 1)^3 T^{-1}
		= 4 - 2T_0 T^{-1}
		\geq 2.
	\end{align*}
	Moreover, we have $\dot \tau(0) = 0$ and $-\dot \tau(T) = \beta T = 2\sup_\Omega |\nabla \phi| + 2$.
	Hence, we can further simplify \eqref{eq:ePv3-PR25} and \eqref{eq:e22F-PR25} as 
	\begin{equation*}
		e^{2\lambda \varphi} |\square u|^2
		\geq \lambda (|\nabla v|^2 + \beta v_t^2) + \lambda^3 v^2 + \partial_t F^0 + \partial_j F^j,
	\end{equation*}
	and
	\begin{equation*}
		\left\{\begin{aligned}
			F^0|_{t=0} & = 4\lambda \nabla \phi \cdot \nabla v v_t + (2n - 4\beta) \lambda v_t v, \\
			F^0|_{t=T} & \geq C(\lambda |\nabla_{t,x} v|^2 + \lambda^3 v^2), \\
			|F^j| & \leq C (\lambda |\nabla_{t,x} v|^2 + \lambda^3 v^2), \qquad 1 \leq j \leq n.
		\end{aligned}\right.
	\end{equation*}
	We see that $F^0|_{t=T} \geq 0$, hence integrating the simplified inequality gives
	\begin{align*}
		& \int_{Q_T} (\beta \lambda |\nabla_{t,x} v|^2 + \lambda^3 v^2) + \int_{\Omega_T} (\lambda |\nabla_{t,x} v|^2 + \lambda^3 v^2) - \int_{\Omega_0} (4\lambda \nabla \phi \cdot \nabla v v_t + (2n - 4\beta) \lambda v_t v) \nonumber \\
		\leq & \int_{Q_T} e^{2\lambda \varphi} |\square u|^2 + C\int_{\Sigma_T} (\lambda |\nabla_{t,x} v|^2 + \lambda^3 v^2).
	\end{align*}
	Transforming $v$ back to $u$ through $u = e^{-\lambda \varphi} v$, and taking $\lambda$ to be large enough, we arrive at the claim.
	For $\mathcal B$, we investigate $F^0|_{t=0}$.
	Using $u = e^{-\lambda \varphi} v$ and $\dot \tau(0) = 0$ we have
	\[
	F^0|_{t=0} = e^{2\lambda \varphi} (4\lambda^2 |\nabla \phi|^2 u_t u + 4\lambda \nabla \phi \cdot \nabla u u_t + (2n - 4\beta) \lambda u_t u),
	\]
	so $\mathcal B = -F^0|_{t=0}$.
	The proof is done.
\end{proof}

\begin{rem} \label{rem:T-PR25}
	In Proposition \ref{prop:InCa-PR25}, by choosing $\varphi$ as $\varphi(t,x) := \frac 1 2 |x - \eta|^2 - \frac \beta 2 t^2$, we see $|\nabla \varphi| = |x - \eta|$, so the condition $\inf_\Omega |\nabla \varphi| \geq 1$ means $\dist (\eta, \Omega) \geq 1$.
	Therefore, by choosing $\eta$ such that $\dist (\eta, \Omega) = 1$, from \eqref{eq:bTl-PR25} we see
	\[
	T_0 = 40 \big(\sup_{x \in \Omega}|x-\eta|+1 \big)^3 \leq 40 (\mathrm{diam}\, \Omega + 2)^3,
	\]
	where $\mathrm{diam}\, \Omega$ signifies the diameter of $\Omega$, i.e., $\mathrm{diam}\, \Omega := \sup_{x,y \in \Omega} \dist(x,y)$.
	Hence, a choice of $T$ could be
	\[
	T = 40 (\mathrm{diam}\, \Omega + 2)^3.
	\]
\end{rem}

\begin{rem}
	Note that in \eqref{eq:bTl-PR25}, $|\nabla \varphi|$ is independent of $\beta$, so the parameters are determined sequentially in the following way:
	$(\Omega, \eta) \rightsquigarrow T_0 \rightsquigarrow T \rightsquigarrow \beta \rightsquigarrow \lambda_0 \rightsquigarrow (\lambda, C_2)$.
\end{rem}

\subsection{Combination of the two estimates}

We first present a well-known Carleman estimate for elliptic operators.
For completeness, we provide a quick proof using Proposition \ref{prop:pwCa-PR25}.

\begin{lem}[Carleman estimate at $t = 0$] \label{lem:pwC2-PR25}
	Fix $x_0 \in \Rn \backslash \overline \Omega$ such that $\dist(x_0, \Omega) \geq 1$.
	There exist two positive constants $\lambda_0$ and $C = C(\lambda_0)$ such that for all $\lambda \geq \lambda_0$ and all $u \in H^2(\Omega)$, there holds
	\begin{equation} \label{eq:ePu3-PR25}
		C \int_\Omega e^{\lambda |x - x_0|^2} (\lambda |\nabla u|^2 + \lambda^3 u^2)
		\leq \int_\Omega e^{\lambda |x - x_0|^2} |\Delta u|^2 + \int_{\partial \Omega} e^{\lambda |x - x_0|^2} (\lambda |\nabla u|^2 + \lambda^3 u^2).
	\end{equation}
\end{lem}

\begin{proof}
	In Proposition \ref{prop:pwCa-PR25}, by setting $a^{jk} = \delta^{jk}$, $\ell = \lambda \varphi$ and $\tilde \ell = 0$, we have
	\begin{equation*}
		e^{2\ell} |\Delta u|^2
		\geq 4 (\varphi_{jk}) \lambda v_j v_k + 2 \nabla \varphi \cdot \nabla (|\nabla \varphi|^2) \lambda^3 v^2 + \mathcal O(|\nabla v|^2 + \lambda^2 v^2) + \partial_j F^j.
	\end{equation*}
	Taking $\varphi = \frac 1 2 |x - x_0|^2$, we obtain
	\begin{align*}
		e^{\lambda |x - x_0|^2} |\Delta u|^2
		& \geq 4 \lambda |\nabla v|^2 + 4|x - x_0|^2 \lambda^3 v^2 + \mathcal O(|\nabla v|^2 + \lambda^2 v^2) + \partial_j F^j \\
		& \geq C(\lambda |\nabla v|^2 + \lambda^3 v^2) + \partial_j F^j.
	\end{align*}
	where the last inequality holds when $\lambda$ is large enough.
	Integrating the inequality in $\Omega$ gives
	\[
	C \int_\Omega (\lambda |\nabla v|^2 + \lambda^3 v^2)
	\leq \int_\Omega e^{\lambda |x - x_0|^2} |\Delta u|^2 + \int_{\partial \Omega} \nu_j F^j.
	\]
	Finally, changing $v$ back to $u$ through $v = e^{-\lambda \varphi} u$ gives the result.
	The derivation originally works for $u \in C^2(\overline \Omega)$, but a density argument extends it to $H^2(\Omega)$.
\end{proof}

When combining the weighted energy estimate with the Carleman estimate, we can cancel the last two terms in \eqref{eq:h0u-PR25} and obtain the following result.

\begin{prop} \label{prop:enCr-PR25}
	Let $M > 0$.
	Under the same assumptions as in Proposition \ref{prop:InCa-PR25}, for any $q \in L^\infty(\Omega)$ with $\norm{q}_{L^\infty(\Omega)} \leq M$, there exist positive constants $\lambda_0 = \lambda_0(\Omega, T, \eta, M)$ and $C = C(T, \lambda_0, M)$ such that for all $\lambda \geq \lambda_0$,
	\begin{align}
		& \ C \Big( \int_{\Omega_0} e^{2\lambda \varphi} (|\nabla_{t,x} u|^2 + \lambda^2 u^2) + \int_{\Omega_0} e^{2\lambda \varphi} \mathcal B \Big) \nonumber \\
		\leq & \ \int_{Q_T} e^{2\lambda \varphi} |(\square + q) u|^2 + \int_{\Sigma_T} e^{2\lambda \varphi} (\lambda |\nabla_{t,x} u|^2 + \lambda^3 u^2), \label{eq:WeC1-PR25}
	\end{align}
	with
	\[
	\mathcal B
	: = u_t \big[ {-}4\lambda^2 |\nabla \varphi|^2 u - 4\lambda \nabla \varphi \cdot \nabla u - (2n - 4\beta) \lambda u \big] \big|_{t=0}.
	\]
	Moreover, with the additional condition $u(0,\cdot) \in H^2(\Omega)$ and
	\begin{equation} \label{eq:WeC2-PR25}
		|\Delta u(0,x)| \leq C(\lambda^\alpha |\nabla u(0,x)| + \lambda^{\alpha'} |u(0,x)|), \quad \alpha \in [0,\frac 1 2), \ \ \alpha' \in [0,\frac 3 2), \ \ x \in \Omega,
	\end{equation}
	we have
	\begin{align}
		& C \int_{\Omega_0} e^{2\lambda \varphi} (|\partial_t u|^2 + \lambda^2 |\nabla u|^2 + \lambda^4 u^2) \nonumber \\
		\leq & \int_{Q_T} e^{2\lambda \varphi} |(\square + q) u|^2 + \int_{\Sigma_T} e^{2\lambda \varphi} (\lambda |\nabla_{t,x} u|^2 + \lambda^3 u^2) + \int_{\partial \Omega_0} e^{2\lambda \varphi} (\lambda^2 |\nabla u|^2 + \lambda^4 u^2). \label{eq:WeC3-PR25}
	\end{align}
\end{prop}

\begin{proof}
	Using the inequality
	\(
	|\square u|^2 \leq 2|(\square + q) u|^2 + 2 M^2 u^2
	\)
	we can generalize the wave operator $\square$ in \eqref{eq:h0u-PR25} and \eqref{eq:InCa-PR25} to $(\square + q) u$.
	
	Combining the two estimates in Propositions \ref{prop:weg1-PR25} \& \ref{prop:InCa-PR25}, we can cancel the integrals on $Q_T$ and $\Omega_T$ to obtain
	\begin{align}
		& \ \int_{\Omega_0} e^{2\lambda \varphi} (|\nabla_{t,x} u|^2 + \lambda^2 u^2) + \beta^{-1} \int_{\Omega_0} e^{2\lambda \varphi} \mathcal B \nonumber \\
		\leq & \ C \Big( \beta^{-1} \int_{Q_T} e^{2\lambda \varphi} |(\square + q) u|^2 + \beta^{-1} \int_{\Sigma_T} e^{2\lambda \varphi} (\lambda |\nabla_{t,x} u|^2 + \lambda^3 u^2) \Big). \label{eq:InC1-PR25}
	\end{align}
	Note that $\beta$ is set to be $2(\sup_\Omega |\nabla \varphi| + 1) T^{-1}$, so it can be absorbed into the generic constant $C$.
	We proved \eqref{eq:WeC1-PR25}.
	
	Note that the term $\beta^{-1} \int_{\Omega_0} e^{2\lambda \varphi} \mathcal B$ in \eqref{eq:InC1-PR25} can be estimate as follows,
	\begin{align*}
		-\big| \beta^{-1} \int_{\Omega_0} e^{2\lambda \varphi} \mathcal B \big|
		& \geq -C \int_{\Omega_0} e^{2\lambda \varphi} |u_t| (\lambda |\nabla u| + \lambda^2 |u|) \\
		& \geq -\frac 1 2 \int_{\Omega_0} e^{2\lambda \varphi} u_t^2 - C \lambda \int_{\Omega_0} e^{2\lambda \varphi} (\lambda |\nabla u|^2 + \lambda^3 u^2).
	\end{align*}
	Combining this with \eqref{eq:InC1-PR25}, we obtain
	\begin{align}
		& \ \int_{\Omega_0} e^{2\lambda \varphi} (|\nabla_{t,x} u|^2 + \lambda^2 u^2) - \lambda \int_{\Omega_0} e^{2\lambda \varphi} (\lambda |\nabla u|^2 + \lambda^3 u^2) \nonumber \\
		\leq & \ C \Big( \int_{Q_T} e^{2\lambda \varphi} |(\square + q) u|^2 + \int_{\Sigma_T} e^{2\lambda \varphi} (\lambda |\nabla_{t,x} u|^2 + \lambda^3 u^2) \Big). \label{eq:InC2-PR25}
	\end{align}
	
	With slight abuse of notation we denote the restriction $u(0,\cdot)$ as $u$ too.
	From \eqref{eq:ePu3-PR25} we have
	\[
	\lambda \int_{\Omega_0} e^{2\lambda \varphi} (\lambda |\nabla u|^2 + \lambda^3 u^2)
	\leq C \lambda \int_{\Omega_0} e^{2\lambda \varphi} |\Delta u|^2 + \lambda \int_{\partial \Omega_0} e^{2\lambda \varphi} (\lambda |\nabla u|^2 + \lambda^3 u^2).
	\]
	Adding this to \eqref{eq:InC2-PR25}, we arrive at
	\begin{align}
		& C \int_{\Omega_0} e^{2\lambda \varphi} (|\partial_t u|^2 + \lambda^2 |\nabla u|^2 + \lambda^4 u^2) \nonumber \\
		\leq & \int_{Q_T} e^{2\lambda \varphi} |(\square + q) u|^2 + \int_{\Sigma_T} e^{2\lambda \varphi} (\lambda |\nabla_{t,x} u|^2 + \lambda^3 u^2) + \int_{\partial \Omega_0} e^{2\lambda \varphi} (\lambda^2 |\nabla u|^2 + \lambda^4 u^2) \nonumber \\
		& + \lambda \int_{\Omega_0} e^{2\lambda \varphi} |\Delta u|^2. \label{eq:InC3-PR25}
	\end{align}
	Compared to \eqref{eq:InC1-PR25}, estimate \eqref{eq:InC3-PR25} gains two orders of $\lambda$ for the left-hand-side terms $|\nabla u|^2$ and $u^2$, but introduces an additional term $\lambda |\Delta u|^2$ on the right-hand side.
	Thanks to the condition \eqref{eq:WeC2-PR25}, the term $\lambda \int_{\Omega_0} e^{2\lambda \varphi} |\Delta u|^2$ can be absorbed by the term $\int_{\Omega_0} e^{2\lambda \varphi} (\lambda^2 |\nabla u|^2 + \lambda^4 u^2)$ on the left-hand-side when $\lambda$ is large, which finally gives \eqref{eq:WeC3-PR25}.
	The proof is done.
\end{proof}

\begin{rem} \label{rem:WeC4-PR25}
	In Proposition \ref{prop:enCr-PR25}, due to the specific form of the boundary term $\mathcal B$, when either $u|_{t = 0} = 0$ or $u_t|_{t = 0} = 0$ in $L^2(\Omega)$, we can immediately conclude from \eqref{eq:WeC1-PR25} that
	\begin{equation*}
		C \int_{\Omega_0} e^{2\lambda \varphi} (|\nabla_{t,x} u|^2 + \lambda^2 u^2)
		\leq \int_{Q_T} e^{2\lambda \varphi} |(\square + q) u|^2 + \int_{\Sigma_T} e^{2\lambda \varphi} (\lambda |\nabla_{t,x} u|^2 + \lambda^3 u^2).
	\end{equation*}
	This is also mentioned in \cite{Klibanov2021book}*{Corollary 2.5.1}.
	Here, our assumption \eqref{eq:WeC2-PR25} is more general.
\end{rem}

\section{Simultaneous recovery} \label{sec:IBVP-PR25}

This section proves Theorem \ref{thm:1-PR25}, i.e., the Simultaneous recovery of the initial conditions and potential.
To prove Theorem \ref{thm:1-PR25}, let us do some preparations.
Let $u_j~(j=1,2)$ be the weak solution of the following PDEs,
\begin{equation} \label{eq:0uj-PR25}
	\left\{\begin{aligned}
		(\square + q_j) u_j & = 0 && \mbox{in} \ \ Q_T, \\
		u_j & = g_j && \mbox{on} \ \ \Sigma_T, \\
		u_j = a_j, \ \partial_t u_j & = b_j && \mbox{on} \ \ \{t = 0\} \times \Omega.
	\end{aligned}\right.
\end{equation}
We first give an estimate for the differences of the initial conditions.

\begin{lem} \label{lem:abes-PR25}
	Let $u_j~(j=1,2)$ be weak solutions of \eqref{eq:0uj-PR25}, with $q_j$, $a_j$, $b_j$, $g_j$ satisfying \eqref{eq:1con-PR25}, and the compatibility conditions up to order $s$.
	Let $M > 0$ be such that $\norm{q_j}_{C^{s+1}(\overline \Omega)} + \norm{a_j}_{H^{s+1}(\Omega)} + \norm{b_j}_{H^s(\Omega)} + \norm{g_j}_{H^{s+1/2}(\Sigma_T)} \leq M$ for both $j=1,2$.
	Set $\varphi$ as in Proposition \ref{prop:InCa-PR25}, then there exist positive constants $\lambda_0$ and $C = C(\beta,T,\Omega,\lambda_0,s,M)$ such that for all $\lambda \geq \lambda_0$, there holds
	\begin{align}
		& C \int_{\Omega_0} e^{2\lambda \varphi} [(b_1 - b_2)^2 + \lambda^2 |\nabla (a_1 - a_2)|^2 + \lambda^4 (a_1 - a_2)^2] \nonumber \\
		\leq & \, \lambda^{-1/2} \int_{\Omega_0} e^{2\lambda \varphi} (q_1 - q_2)^2 + \int_{\Sigma_T} e^{2\lambda \varphi} [\lambda |\nabla_{t,x} (u_1 - u_2)|^2 + \lambda^3 (u_1 - u_2)^2] \nonumber \\
		& + \int_{\partial \Omega_0} e^{2\lambda \varphi} [\lambda^2 |\nabla (u_1 - u_2)|^2 + \lambda^4 (u_1 - u_2)^2], \label{eq:abes-PR25}
	\end{align}
	provided that there exist two constants $\alpha \in [0,\frac 1 2)$ and $\alpha' \in [0,\frac 3 2)$ such that
	\begin{equation} \label{eq:abs-PR25}
		|\Delta (a_1 - a_2)| \leq C(\lambda^\alpha |\nabla (a_1 - a_2)| + \lambda^{\alpha'} |a_1 - a_2|) \quad \text{in} \quad \Omega.
	\end{equation}
\end{lem}

\begin{proof}
	Let $u := u_1 - u_2$, then $u$ satisfies
	\begin{equation} \label{eq:0u-PR25}
		\left\{\begin{aligned}
			(\square + q_1) u & = (q_2 - q_1) u_2 && \mbox{in} \ \ Q_T, \\
			u & = 0 && \mbox{on} \ \ \Sigma_T, \\
			u = a_1 - a_2, \ \partial_t u & = b_1 - b_2 && \mbox{on} \ \ \{t = 0\} \times \Omega.
		\end{aligned}\right.
	\end{equation}
	Since $\norm{q_1}_{L^\infty(\Omega)} \leq \norm{q_1}_{C^{s+1}(\overline \Omega)} \leq M$, Proposition \ref{prop:enCr-PR25} applies with the constant $C(T,\lambda_0,M)$.
	Utilizing the condition \eqref{eq:abs-PR25}, we can apply Proposition \ref{prop:enCr-PR25} to \eqref{eq:0u-PR25} to have
	\begin{align}
		& C \int_{\Omega_0} e^{2\lambda \varphi} [(b_1 - b_2)^2 + \lambda^2 |\nabla (a_1 - a_2)|^2 + \lambda^4 (a_1 - a_2)^2] \nonumber \\
		\leq & \int_{Q_T} e^{2\lambda \varphi} |(q_2 - q_1) u_2|^2 + \int_{\Sigma_T} e^{2\lambda \varphi} (\lambda |\nabla_{t,x} u|^2 + \lambda^3 u^2) \nonumber \\
		& + \int_{\partial \Omega_0} e^{2\lambda \varphi} (\lambda^2 |\nabla u|^2 + \lambda^4 u^2). \label{eq:ab13-PR25}
	\end{align}
	
	By the assumed regularities and Lemma \ref{lem:Linf-PR25}, we can conclude there exists a number $C(T,\Omega,s,M)$ such that
	\(
	\norm{u_2}_{L^\infty(Q_T)} \leq C(T,\Omega,s,M),
	\)
	so
	\begin{align}
		\int_{Q_T} e^{2\lambda \varphi} |(q_2 - q_1) u_2|^2
		& = \int_{\Omega} \Big( \int_0^T e^{2\lambda (\varphi(t) - \varphi(0))} |u_2|^2 \dif t \Big) e^{2\lambda \varphi(0)} (q_1 - q_2)^2 \dif x \nonumber \\
		& \leq C(T,\Omega,s,M) \int_{\Omega} \Big( \int_0^T e^{-\lambda \beta t^2} \dif t \Big) e^{2\lambda \varphi(0)} (q_1 - q_2)^2 \dif x \nonumber \\
		& \leq C(T,\Omega,s,M) \, (\beta \lambda)^{-1/2} \int_{\Omega_0} e^{2\lambda \varphi} (q_1 - q_2)^2 \dif x. \label{eq:qT0-PR25}
	\end{align}
	Substituting this into \eqref{eq:ab13-PR25}, we arrive at \eqref{eq:abes-PR25}.
\end{proof}

We are ready to prove Theorem \ref{thm:1-PR25}.

\begin{proof}[Proof of Theorem \ref{thm:1-PR25}]
	Let $u_j~(j=1,2)$ satisfy \eqref{eq:1-PR25} with $q_j$ and denote $v := \partial_t (u_1 - u_2)$, then $v$ satisfies
	\begin{equation} \label{eq:0v-PR25}
		\left\{\begin{aligned}
			(\square + q_1) v & = (q_2 - q_1) \partial_t u_2 && \mbox{in} \ \ Q_T, \\
			v & = \partial_t(g_1 - g_2) && \mbox{on} \ \ \Sigma_T, \\
			v = b_1 - b_2, \ \partial_t v & = (q_2 - q_1) a_2 + (\Delta - q_1) (a_1 - a_2) && \mbox{on} \ \ \{t = 0\} \times \Omega.
		\end{aligned}\right.
	\end{equation}
	with the condition \eqref{eq:as1b-PR25}, applying \eqref{eq:WeC3-PR25} to \eqref{eq:0v-PR25}, we can have the following estimate,
	\begin{align*}
		& C \int_{\Omega_0} e^{2\lambda \varphi} \big\{ [(q_2 - q_1) a_2 + (\Delta - q_1) (a_1 - a_2)]^2 + \lambda^2 |\nabla (b_1 - b_2)|^2 + \lambda^4 (b_1 - b_2)^2 \big\} \\
		\leq & \int_{Q_T} e^{2\lambda \varphi} |(q_1 - q_2) \partial_t u_2|^2 + \int_{\Sigma_T} e^{2\lambda \varphi} (\lambda |\nabla_{t,x} v|^2 + \lambda^3 v^2) + \int_{\partial \Omega_0} e^{2\lambda \varphi} (\lambda^2 |\nabla v|^2 + \lambda^4 v^2).
	\end{align*}
	Using the basic algebra $(x+y)^2 \geq \frac 1 2 x^2 - y^2$ and the conditions \eqref{eq:as1a-PR25} and \eqref{eq:ass2-PR25}, we know
	\begin{align*}
		& \, [(q_2 - q_1) a_2 + (\Delta - q_1) (a_1 - a_2)]^2 \\
		\geq & \, \frac {C_2^2} 2 (q_1 - q_2)^2 - (4\widetilde C^2 + 2\norm{q_1}_{L^\infty(\Omega)}^2) [(a_1 - a_2)^2 + |\nabla (a_1 - a_2)|^2],
	\end{align*}
	thus the estimate becomes
	\begin{align*}
		& C \int_{\Omega_0} e^{2\lambda \varphi} [(q_1 - q_2)^2 + \lambda^2 |\nabla (b_1 - b_2)|^2 + \lambda^4 (b_1 - b_2)^2] \nonumber \\
		\leq & \int_{Q_T} e^{2\lambda \varphi} |(q_1 - q_2) \partial_t u_2|^2 + \int_{\Sigma_T} e^{2\lambda \varphi} (\lambda |\nabla_{t,x} v|^2 + \lambda^3 v^2) + \int_{\partial \Omega_0} e^{2\lambda \varphi} (\lambda^2 |\nabla v|^2 + \lambda^4 v^2) \nonumber \\
		& + \int_{\Omega_0} e^{2\lambda \varphi} \big[ (a_1 - a_2)^2 + |\nabla (a_1 - a_2)|^2 \big],
	\end{align*}
	for some $C=C(T,\Omega,s,M,C_2,\widetilde C,\varphi)$.
	
	Moreover, since \(u_2\in H^{s+1}(Q_T)\), we have \(\partial_t u_2\in H^s(Q_T)\).
	Because \(s \geq \lceil n/2 \rceil + 1 >(n+1)/2\), Sobolev embedding gives \(\norm{\partial_t u_2}_{L^\infty(Q_T)} \leq C \norm{\partial_t u_2}_{H^s(Q_T)} \leq C(T,\Omega,s,M)\), thus similar to \eqref{eq:qT0-PR25} we can have
	\[
	\int_{Q_T} e^{2\lambda \varphi} |(q_1 - q_2) \partial_t u_2|^2 \dif t \dif x
	\leq C(T,\Omega,s,M) \, (\beta \lambda)^{-1/2} \int_{\Omega_0} e^{2\lambda \varphi} (q_1 - q_2)^2 \dif x.
	\]
	Therefore, the estimate further becomes
	\begin{align}
		& \, C(\beta, \Omega, T, \lambda_0, s, M) \int_{\Omega_0} e^{2\lambda \varphi} [(q_1 - q_2)^2 + \lambda^2 |\nabla (b_1 - b_2)|^2 + \lambda^4 (b_1 - b_2)^2] \nonumber \\
		\leq & \, \lambda^{-1/2} \int_{\Omega_0} e^{2\lambda \varphi} (q_1 - q_2)^2 + \int_{\Sigma_T} e^{2\lambda \varphi} (\lambda |\nabla_{t,x} v|^2 + \lambda^3 v^2) + \int_{\partial \Omega_0} e^{2\lambda \varphi} (\lambda^2 |\nabla v|^2 + \lambda^4 v^2) \nonumber \\
		& \, + \int_{\Omega_0} e^{2\lambda \varphi} \big[ (a_1 - a_2)^2 + |\nabla (a_1 - a_2)|^2 \big]. \label{eq:ab14-PR25}
	\end{align}
	
	Since Assumption \ref{asm:assu-PR25} gives \eqref{eq:abs-PR25},
	Lemma \ref{lem:abes-PR25} can be applied in this proof, and we use it to deal with the last term in \eqref{eq:ab14-PR25}:
	\begin{align*}
		& C(\beta, \Omega, T, \lambda_0, s, M) \int_{\Omega_0} e^{2\lambda \varphi} [\lambda^2 |\nabla (a_1 - a_2)|^2 + \lambda^4 (a_1 - a_2)^2] \nonumber \\
		\leq & \, \lambda^{-1/2} \int_{\Omega_0} e^{2\lambda \varphi} (q_1 - q_2)^2 + \int_{\Sigma_T} e^{2\lambda \varphi} (\lambda |\nabla_{t,x} u|^2 + \lambda^3 u^2) + \int_{\partial \Omega_0} e^{2\lambda \varphi} (\lambda^2 |\nabla u|^2 + \lambda^4 u^2),
	\end{align*}
	where $u = u_1 - u_2$.
	Summing up this with \eqref{eq:ab14-PR25} and choosing $\lambda\ge\lambda_0(T,\Omega,s,M)$ large enough so that the coefficient of the absorbed term is less than $1/2$,
	and absorbing $\lambda^{-1/2} \int_{\Omega_0} e^{2\lambda \varphi} (q_1 - q_2)^2$ to the left, we arrive at
	\begin{align}
		& C \int_{\Omega_0} e^{2\lambda \varphi} [(q_2 - q_1)^2 + \lambda^2 |\nabla (a_1 - a_2)|^2 + \lambda^4 (a_1 - a_2)^2 + \lambda^2 |\nabla (b_1 - b_2)|^2 + \lambda^4 (b_1 - b_2)^2] \nonumber \\
		\leq & \int_{\Sigma_T} e^{2\lambda \varphi} (\lambda |\nabla_{t,x} v|^2 + \lambda^3 v^2) + \int_{\partial \Omega_0} e^{2\lambda \varphi} (\lambda^2 |\nabla v|^2 + \lambda^4 v^2) \nonumber \\
		& + \int_{\Sigma_T} e^{2\lambda \varphi} (\lambda |\nabla_{t,x}  u|^2 + \lambda^3 u^2) + \int_{\partial \Omega_0} e^{2\lambda \varphi} (\lambda^2 |\nabla u|^2 + \lambda^4 u^2). \label{eq:abq-PR25}
	\end{align}
	This implies
	\begin{align*}
		& \, \norm{q_1 - q_2}_{L^2(\Omega)} + \norm{a_1 - a_2}_{H^1(\Omega)} + \norm{b_1 - b_2}_{H^1(\Omega)} \nonumber \\
		\leq & \, C \big( \norm{u_1 - u_2}_{H^2(\Sigma_T)} + \norm{\partial_\nu (u_1 - u_2)}_{H^1(\Sigma_T)} \nonumber \\
		& \, + \norm{\partial_t (u_1 - u_2)}_{H^1(\partial \Omega)} + \norm{\partial_\nu \partial_t (u_1 - u_2)}_{L^2(\partial \Omega)} \nonumber \\
		& \, + \norm{u_1 - u_2}_{H^1(\partial \Omega)} + \norm{\partial_\nu (u_1 - u_2)}_{L^2(\partial \Omega)} \big).
	\end{align*}
	Since $\varphi$ is fixed and bounded on $Q_T$, the weight factors are absorbed into the number $C$.
	Noting that $\partial \Omega$ is identified with the edge $\{0\} \times \partial \Omega$ of $\Sigma_T$, so by Sobolev trace theorem, the above estimate can be simplified to
	\begin{align*}
		& \, \norm{q_1 - q_2}_{L^2(\Omega)} + \norm{a_1 - a_2}_{H^1(\Omega)} + \norm{b_1 - b_2}_{H^1(\Omega)} \nonumber \\
		\leq & \, C \big( \norm{u_1 - u_2}_{H^3(\Sigma_T)} + \norm{\partial_\nu (u_1 - u_2)}_{H^2(\Sigma_T)} \big).
	\end{align*}
	This is the desired result.
\end{proof}

\section{Initial-potential problem} \label{sec:schm-PR25}

This section studies the initial-potential problem \eqref{eq:2-PR25}, in which the initial displacement is the unknown potential and the initial velocity vanishes.
We first establish the Lipschitz stability of Theorem \ref{thm:2-PR25}, then propose an iterative reconstruction scheme and show its geometric convergence.

\subsection{Lipschitz stability} \label{subsec:sSta-PR25}

We are ready to prove Theorem \ref{thm:2-PR25}.

\begin{proof}[Proof of Theorem \ref{thm:2-PR25}]
	Let $u_j~(j=1,2)$ satisfy \eqref{eq:2-PR25} with $q_j$ and $g_j$.
	Denote $u := u_1 - u_2$, then $u$ satisfies
	\begin{equation*}
		\left\{\begin{aligned}
			(\square + q_1) u & = (q_2 - q_1) u_2 && \mbox{in} \ \ Q_T, \\
			u & = g_1 - g_2 && \mbox{on} \ \ \Sigma_T, \\
			u & = q_1 - q_2, \ \partial_t u = 0 && \mbox{on} \ \ \{t = 0\} \times \Omega.
		\end{aligned}\right.
	\end{equation*}
	Since $\norm{q_1}_{L^\infty(\Omega)} \leq \norm{q_1}_{C^{s+1}(\overline \Omega)} \leq M$, Proposition \ref{prop:enCr-PR25} applies with the constant $C(T,\lambda_0,M)$.
	By Remark \ref{rem:WeC4-PR25} we can have
	\begin{equation*}
		C \int_{\Omega_0} e^{2\lambda \varphi} [|\nabla (q_1 - q_2)|^2 + \lambda^2 (q_1 - q_2)^2]
		\leq \int_{Q_T} e^{2\lambda \varphi} |(q_2 - q_1) u_2|^2 + \int_{\Sigma_T} e^{2\lambda \varphi} (\lambda |\nabla_{t,x} u|^2 + \lambda^3 u^2).
	\end{equation*}
	By the assumed regularities $g_j \in H^{s+1/2}(\Sigma_T)$, $q_j \in C^{s+1}(\overline \Omega)$ and Corollary \ref{cor:Linf-PR25}, we have
	\begin{equation*}
		\norm{u_j}_{L^\infty(Q_T)}
		\leq C(T,\Omega,s) \,
		e^{C ( 1 + \norm{q_j}_{C^{s+1}} ) T} (\norm{g_j}_{H^{s+1/2}(\Sigma_T)} + \norm{q_j}_{H^{s+1}(\Omega)})
		\leq C(T,\Omega,s,M),
	\end{equation*}
	so just as in \eqref{eq:qT0-PR25} we have
	\begin{equation} \label{eq:qT1-PR25}
		\int_{Q_T} e^{2\lambda \varphi} |(q_2 - q_1) u_2|^2
		\leq C(T,\Omega,s,M) \, (\beta \lambda)^{-1/2} \int_{\Omega_0} e^{2\lambda \varphi} (q_1 - q_2)^2 \dif x.
	\end{equation}
	Plugging \eqref{eq:qT1-PR25} into the right-hand side of the inequality
	above (which follows from Remark~\ref{rem:WeC4-PR25}), and absorbing the weighted $L^2$-norm of $q_1 - q_2$ for \(\lambda \geq \lambda_0(T,\Omega,s,M)\), we deduce
	\begin{equation} \label{eq:WeC4-PR25}
		C \int_{\Omega_0} e^{2\lambda \varphi} [|\nabla (q_1 - q_2)|^2 + \lambda^2 (q_1 - q_2)^2]
		\leq \int_{\Sigma_T} e^{2\lambda \varphi} (\lambda |\nabla_{t,x} u|^2 + \lambda^3 u^2).
	\end{equation}
	This implies the result.
\end{proof}

\subsection{Convex optimization preliminaries}

After proving the stability result of Theorem \ref{thm:2-PR25}, we turn to the reconstruction of the potential.
The reconstruction is an optimization procedure, so we first record two standard facts from convex analysis.
The first ensures that the minimizers used below are well defined and unique.
The second ensures that the projected gradient iteration used to compute them converges.
A good reference would be \cite[Chapter 28]{Bauschke2017Convex}.

\begin{lem} \label{lem:mini-PR25}
	Let $\mathcal H$ be a Hilbert space, and $\mathcal A \subset \mathcal H$ is a nonempty, bounded and closed convex subset.
	Assume $f \colon \mathcal H \to \mathbb R$ is weakly lower semicontinuous on $\mathcal H$
	and strictly convex on $\mathcal A$.
	Assume further $f$ is bounded below on $\mathcal A$.
	Then $f$ attains its minimum on $\mathcal A$ at a unique point.
\end{lem}

\begin{proof}
	We first show the existence.
	Denote $m := \inf_{x \in \mathcal A} f(x)$.
	Since $f$ is bounded below and $\mathcal A$ is nonempty, $m$ is finite.
	By the definition of the infimum, we can find a sequence $\{x_k\}_{k \geq 0} \subset \mathcal A$ such that $\lim_{k \to \infty} f(x_k) = m$.
	The set $\mathcal A$ is closed and bounded and convex, so it is weakly compact, thus there exist a subsequence $\{x_{k_\ell}\}_{\ell \ge 0}$ of $\{x_k\}_{k \ge 0}$ and $x_* \in \mathcal A$ such that $x_{k_\ell} \rightharpoonup x_*$ in $\mathcal H$.
	Since $f$ is weakly lower semicontinuous,
	\[
	f(x_*)
	\leq \liminf_{\ell \to \infty} f(x_{k_\ell})
	= \lim_{k \to \infty} f(x_k)
	= m.
	\]
	Since $m$ is the infimum, we can conclude $f(x_*) = m$.
	
	As for the uniqueness, if $x,y \in \mathcal A$ are minimizers with $x \neq y$, then $(x+y)/2 \in \mathcal A$, and strict convexity of $f$ gives $f((x+y)/2) < (f(x) + f(y))/2 = m$, a contradiction.
\end{proof}

The second fact concerns the projected gradient method. We work in the following setting.
Let $\mathcal H$ be a real Hilbert space with inner product $\langle \cdot, \cdot \rangle$ and norm $\norm{\cdot}$.
Let $\mathcal A \subset \mathcal H$ be nonempty, closed, convex, and bounded with bound $M_{\mathcal A}$.
Let $f: \mathcal H \to \mathbb R$ be convex and Fr\'echet differentiable with Fr\'echet derivative denoted as $Df(x)$.
Since $Df(x) \in \mathcal H^*$, the Riesz representation theorem identifies $Df(x)$ with an element of $\mathcal H$, denoted $f'(x)$, so that $Df(x)v = \langle f'(x), v \rangle$ for all $v \in \mathcal H$.
Assume that $f'$ is $L$-Lipschitz:
\begin{equation*}
	\norm{f'(x) - f'(y)} \leq L \norm{x - y}, \qquad x, y \in \mathcal H.
\end{equation*}
Assume that $f$ is bounded below on $\mathcal A$.
Fix $\tau \in (0, 2L^{-1})$ and $x_0 \in \mathcal A$, and set
\begin{equation} \label{eq:aIter-PR25}
	x_{k+1} := \mathcal P_{\mathcal A} \bigl(x_k - \tau f'(x_k)\bigr), \qquad k = 0, 1, 2, \dots,
\end{equation}
where $\mathcal P_{\mathcal A}$ is the metric projection onto $\mathcal A$.
In this setting, the projected gradient iteration converges to the unique minimizer whenever the objective is strictly convex.

\begin{lem} \label{lem:app-PR25}
	Let $\mathcal H$ be a Hilbert space, and $\mathcal A \subset \mathcal H$ is a nonempty, bounded and closed convex subset.
	Assume $f \colon \mathcal H \to \mathbb R$ is convex and weakly lower semicontinuous on $\mathcal H$, and strictly convex on $\mathcal A$.
	Assume further $f$ is bounded below on $\mathcal A$, and $f \in C^{1,1}$ in the sense that $f$ is Fr\'echet differentiable and $f'$ satisfies
	\begin{equation} \label{eq:aLip-PR25}
		\norm{f'(x) - f'(y)} \leq L \norm{x - y}, \quad \forall x,y \in \mathcal H,
	\end{equation}
	for certain constant $L > 0$.
	When the number $\tau$ in \eqref{eq:aIter-PR25} satisfies $0 < \tau < 2L^{-1}$, the sequence $\{x_k\}_{k \geq 0}$ defined in \eqref{eq:aIter-PR25} with arbitrary starting point $x_0 \in \mathcal A$ is weakly convergent to the unique minimizer $x^*$ of $f$ on $\mathcal A$,
	and $\lim_{k \to \infty} f(x_k) = \min_{x \in \mathcal A} f(x)$.
\end{lem}

\begin{proof}
	We first record two standard estimates for a strictly convex function $f \in C^{1,1}$:
	\begin{align}
		\langle z - \mathcal P_{\mathcal A}(z),\, y - \mathcal P_{\mathcal A}(z) \rangle & \leq 0, && z \in \mathcal H, \ y \in \mathcal A. \label{eq:aPrj-PR25} \\
		0 \leq f(x) - f(y) - \langle f'(y), x - y \rangle
		& \leq \frac L2 \norm{x - y}^2, && x, y \in \mathcal H, \label{eq:aDes-PR25}
	\end{align}
	where \eqref{eq:aPrj-PR25} comes from the variational characterization of the projection, and the first inequality in \eqref{eq:aDes-PR25} comes from the convexity of $f$:
	\[
	\frac {\df} {\df t} \Big|_{t=0} \Big( tf(x) + (1-t) f(y) - f(tx + (1-t)y) \Big) \geq 0,
	\]
	and the second one comes from the fundamental theorem of calculus and \eqref{eq:aLip-PR25}:
	\[
	f(x) - f(y) - \langle f'(y), x - y \rangle
	= \int_0^1 \langle f'(y + t(x - y)) - f'(y),\, x - y \rangle \dif t
	\leq \frac L2 \norm{x - y}^2.
	\]

	Taking $z = x_k - \tau f'(x_k)$ in \eqref{eq:aPrj-PR25}, we obtain
	\begin{equation} \label{eq:aMon-PR25}
		\langle f'(x_k),\, y - x_{k+1} \rangle
		\geq \tau^{-1} \langle x_k - x_{k+1},\, y - x_{k+1} \rangle, \qquad y \in \mathcal A.
	\end{equation}
	where $x_{k+1} := \mathcal P_{\mathcal A}(z)$ as already defined in \eqref{eq:aIter-PR25}.
	Setting $x = x_{k+1}$ and $y = x_k$ in both \eqref{eq:aDes-PR25} and \eqref{eq:aMon-PR25} and summing up together, we obtain
	\begin{equation}  \label{eq:aDec-PR25}
		(\frac 1 \tau - \frac L 2) \norm{x_{k+1} - x_k}^2
		\leq f(x_k) - f(x_{k+1}).
	\end{equation}
	Fix $\tau$ such that $0 < \tau < 2L^{-1}$, then \eqref{eq:aDec-PR25} implies the sequence $\{f(x_k)\}_{k \geq 0}$ is non-increasing and so has a limit
	\[
	f_\infty := \lim_{k \to \infty} f(x_k),
	\]
	and
	\begin{equation} \label{eq:xxkk-PR25}
		\left\{\begin{aligned}
			& (\frac 1 \tau - \frac L 2) \sum_{k=0}^{\infty} \norm{x_{k+1} - x_k}^2
			\leq f(x_0) - f_\infty < \infty, \\
			& \lim_{k \to \infty} \norm{x_{k+1} - x_k} = 0.
		\end{aligned}\right.
	\end{equation}
	
	Using again \eqref{eq:aDes-PR25} but with $x \in \mathcal A$ and $y = x_{k+1} \in \mathcal A$ we have
	\begin{align*}
		f(x)
		& \geq f(x_{k+1}) + \langle f'(x_{k+1}), x - x_{k+1} \rangle \\
		& \geq f(x_{k+1}) + \langle f'(x_{k+1}) - f'(x_k), x - x_{k+1} \rangle + \langle f'(x_k), x - x_{k+1} \rangle \\
		& \geq f(x_{k+1}) - \norm{f'(x_{k+1}) - f'(x_k)} \norm{x - x_{k+1}} + \tau^{-1} \langle x_k - x_{k+1},\, x - x_{k+1} \rangle \\
		& \geq f(x_{k+1}) - 2 M_{\mathcal A} (L + \tau^{-1}) \norm{x_{k+1} - x_k}, \quad \forall x \in \mathcal A,
	\end{align*}
	where we used \eqref{eq:aLip-PR25} and \eqref{eq:aMon-PR25}.
	Taking the limit, from \eqref{eq:xxkk-PR25} we see $f(x) \geq f_\infty$ holds for all $x \in \mathcal A$, which implies
	\begin{equation} \label{eq:fity-PR25}
		f_\infty = \inf_{x \in \mathcal A} f(x).
	\end{equation}
	
	Since $\mathcal A$ is bounded, $\{x_k\}_{k \geq 0}$ is bounded.
	Let $\{x_{k_\ell}\}_{\ell \geq 0}$ be one of the weakly convergent subsequence of $\{x_k\}_{k \geq 0}$ such that $x_{k_\ell} \rightharpoonup x_\infty$ for certain $x_\infty \in \mathcal H$.
	The set $\mathcal A$ is closed and convex, thus weakly closed, so $x_\infty \in \mathcal A$.

	The functional $f$ is weakly lower semicontinuous, so $f(x_\infty) \leq \liminf_{\ell \to \infty} f(x_{k_\ell})$, thus by \eqref{eq:fity-PR25},
	\begin{equation*}
		f(x_\infty)
		\leq \liminf_{\ell \to \infty} f(x_{k_\ell})
		= \lim_{\ell \to \infty} f(x_{k_\ell})
		= \lim_{k \to \infty} f(x_k)
		= f_\infty
		= \inf_{x \in \mathcal A} f(x),
	\end{equation*}
	which, together with the reverse direction $f(x_\infty) \geq \inf_{x \in \mathcal A} f(x)$, implies
	\begin{equation*}
		f(x_\infty) = \inf_{x \in \mathcal A} f(x),
	\end{equation*}
	thus $f$ attains its minimum at the point $x_\infty$, and $\lim_{k \to \infty} f(x_k) = \min_{x \in \mathcal A} f(x)$.
	By Lemma \ref{lem:mini-PR25}, $f$ attains its minimum on $\mathcal A$ at a unique point $x_*\in\mathcal A$.
	Since $x_\infty$ is a minimizer, $x_\infty=x_*$.
	Every weak cluster point of $\{x_k\}$ is a minimizer by the same argument, hence equals $x_*$.
	Since $\{x_k\}$ is bounded, it converges weakly to $x_*$.
\end{proof}

These two lemmas are the abstract ingredients of the reconstruction.
The first provides the unique minimizer of the functional defined below, and the second justifies the inner iteration that computes it.

\subsection{The reconstruction scheme} \label{subsec:sSet-PR25}

Let $u^*$ be the solution of the initial-potential problem
\begin{equation} \label{eq:2str-PR25}
	\left\{\begin{aligned}
		(\square + q^*) u^* & = 0, && \text{in} \ Q_T, \\
		u^* & = g^*, && \text{on} \ \Sigma_T, \\
		u^* = q^*, \ \partial_t u^* & = 0, && \text{in} \ \{0\} \times \Omega,
	\end{aligned}\right.
\end{equation}
with $q^*$ and $g^*$ satisfying the regularity assumptions \eqref{eq:1con-PR25}, $s \geq \max\{\lceil n/2\rceil + 1, 3\}$, and the compatibility conditions up to order $s$.
Then by e.g.~\cite{EvanPDEs}*{Chapter 7} we have that $u^* \in H^2(Q_T)$.

Let $M > 0$ be such that
\begin{equation} \label{eq:gpM-PR25}
	\norm{g^*}_{H^{s+1/2}(\Sigma_T)} + \norm{q^*}_{C^{s+1}(\overline \Omega)} \leq M.
\end{equation}

We define the following admissible set
\begin{equation} \label{eq:sAdm-PR25}
	\mathcal A := \Big\{u \in H^2(Q_T):\ u|_{\Sigma_T} = g^*, \ \partial_t u(0, \cdot) = 0, \
	\norm{u}_{H^2(Q_T)} \leq 2(\norm{u^*}_{H^2(Q_T)} + 1) \Big\}.
\end{equation}
It is non-empty because $u^* \in \mathcal A$.
Moreover, it is easy to see $\mathcal A$ is a convex set.
Recall the Carleman phase function $\varphi = \tfrac 1 2 |x-\eta|^2 - \tfrac \beta 2 t^2$ defined in Proposition \ref{prop:InCa-PR25}.
The initial guess can be any $q_0 \in L^\infty(\Omega)$ satisfying $\norm{q_0}_{L^\infty(\Omega)} \leq M$.
For simplicity, we set $q_0 = 0.$
For every integer $i \geq 0$ we define the following functional:
\begin{equation} \label{eq:sFun-PR25}
	\mathcal J_i(u) := \int_{Q_T} e^{2\lambda\varphi} |(\square + q_i) u|^2 \dif x \dif t + \lambda \int_{\Sigma_T} e^{2\lambda\varphi} |\partial_\nu u - \partial_\nu u^*|^2 \dif \sigma, \quad u \in H^2(Q_T).
\end{equation}
Define
\begin{equation} \label{eq:sUpd-PR25}
	u_i := \operatorname*{arg\,min}_{u \in \mathcal A} \mathcal J_i(u), \quad
	q_{i+1} := T_M(u_i(0, \cdot)), \quad i \geq 0,
\end{equation}
where $T_M f$ is the truncated version of $f$ such that
\begin{equation} \label{eq:TM-PR25}
	T_M f(x) :=
	\left\{\begin{aligned}
		f(x), \quad & |f(x)| \leq M, \\
		\pm M, \quad & \pm f(x) > M.
	\end{aligned}\right.
\end{equation}
Therefore, $\norm{T_M f}_{L^\infty(\Omega)} \leq M$ is always true.
It can also be verified that the pointwise bound
\(
|T_M f - T_M g| \leq |f-g|
\)
holds everywhere.

Our reconstruction scheme is as follows:
\begin{enumerate}
	\item Set $q_0 = 0$ and $u_0^{(0)} = \mathcal P_{\mathcal A}(0)$, where $\mathcal P_{\mathcal A}$ is the metric projection onto $\mathcal A$.
	
	\item Fix $i \geq 0$, and use $u_i^{(j + 1)} = \mathcal P_{\mathcal A} \big( u_i^{(j)} - \tau \mathcal J_i'(u_i^{(j)}) \big)$ to construct $\{u_i^{(j)}\}_{j \geq 0}$.
	See \eqref{eq:pg-PR25}.
	
	\item Set $q_{i+1} := \lim_{j \to \infty} T_M \big( u_i^{(j)}(0,\cdot) \big)$.
	The limit converges in $L^2(\Omega)$. See \eqref{eq:uq-PR25}.
	
	\item Find the weak limit $u_i$ of $\{u_i^{(j)}\}_{j \geq 0}$. It is the unique minimizer of $\mathcal J_i$ in $\mathcal A$.
	
	\item Increase $i$ by $1$, set $u_i^{(0)} = u_{i-1}$, and repeat Steps (2) and (3) to construct $\{ q_i \}_{i \geq 0}$.
\end{enumerate}
The constructed sequence $\{ q_i \}_{i \geq 0}$ converges to $q^*$ in $L^2(\Omega)$, see Corollary \ref{cor:L2-PR25}.

The projected gradient iteration is a standard instance of the forward-backward splitting for averaged operators; existence of weak convergence and objective convergence follows from the classical theory in \cite[Chapter 28]{Bauschke2017Convex}.

\subsection{Convexity and iteration} \label{subsec:sWel-PR25}

The constants in Proposition \ref{prop:enCr-PR25} are uniform over potentials with $\norm{q}_{L^\infty(\Omega)} \leq M$, which is exactly the class used below.

\begin{lem} \label{lem:ker-PR25}
	Let $\norm{q}_{L^\infty(\Omega)} \leq M$, and let $w \in H^2(Q_T)$ satisfy
	\begin{equation*}
		\left\{\begin{aligned}
			(\square + q) w & = 0, && \text{in} \ Q_T, \\
			w = 0, \ \partial_\nu w & = 0, && \text{on} \ \Sigma_T, \\
			\partial_t w & = 0, && \text{in} \ \{0\} \times \Omega.
		\end{aligned}\right.
	\end{equation*}
	Then $w = 0$ in $Q_T$.
\end{lem}

\begin{proof}
	Apply Proposition \ref{prop:enCr-PR25} to $w$, we have
	\begin{equation*}
	C \int_{\Omega_0} e^{2\lambda \varphi} (|\nabla_{t, x} w|^2 + \lambda^2 w^2)
	\leq \int_{Q_T} e^{2\lambda \varphi} |(\square + q) w|^2 + \int_{\Sigma_T} e^{2\lambda \varphi} (\lambda |\nabla_{t, x} w|^2 + \lambda^3 w^2)
	= 0,
	\end{equation*}
	so $w(0, \cdot) = 0$.
	The wave equation with zero initial data and zero lateral Dirichlet data gives $w = 0$ in $Q_T$ by energy uniqueness.
\end{proof}

\begin{lem} \label{lem:exst-PR25}
	Let $\norm{q_i}_{L^\infty(\Omega)} \leq M$.
	Then $\mathcal J_i$ is convex on $H^2(Q_T)$ and strictly convex on $\mathcal A$, and it attains its minimum on $\mathcal A$ at a unique point $u_i$.
\end{lem}

\begin{proof}
	In this proof we fix $i$ for simplicity.
	From \eqref{eq:sAdm-PR25} it can be checked that $\mathcal A \subset H^2(Q_T)$ is nonempty, bounded, closed and convex.
	
	The convexity of $\mathcal J_i$ on $H^2(Q_T)$ is obvious.
	For the strict convexity of $\mathcal J_i$, let $u_1, u_2 \in \mathcal A$, and let $\mathcal Q_i(w)$ be the quadratic part of $\mathcal J_i$:
	\[
	\mathcal Q_i(w) := \int_{Q_T} e^{2\lambda\varphi}|(\square + q_i) w|^2\dif x \dif t
	+ \lambda\int_{\Sigma_T} e^{2\lambda\varphi}|\partial_\nu w|^2\dif\sigma.
	\]
	Direct computation shows that
	\[
	s\mathcal J_i(u_1) + (1 - s)\mathcal J_i(u_2) - \mathcal J_i(su_1 + (1 - s)u_2)
	= s(1 - s)\mathcal Q_i(u_1 - u_2), \quad 0 \leq s \leq 1.
	\]
	Note that the convexity of $\mathcal A$ gives $su_1 + (1 - s)u_2 \in \mathcal A$.
	Denote $w := u_1 - u_2$ for short.
	If $\mathcal Q_i(w) = 0$, then $(\square + q_i) w = 0$ in $Q_T$ and $\partial_\nu w = 0$ on $\Sigma_T$, and $w = 0$ on $\Sigma_T$ and $\partial_t w(0, \cdot) = 0$ in $\Omega$ by the definition \eqref{eq:sAdm-PR25}, so \eqref{eq:WeC1-PR25} in Proposition \ref{prop:enCr-PR25} implies $w(0,\cdot) = 0$ in $\Omega$, and energy uniqueness then further gives $w = 0$ in $Q_T$.
	Hence $\mathcal Q_i(w) > 0$ for every $u_1 \neq u_2$ in $\mathcal A$.
	Therefore, for $0 < s < 1$ and distinct $u_1, u_2 \in \mathcal A$, we have
	\[
	s\mathcal J_i(u_1) + (1 - s)\mathcal J_i(u_2) > \mathcal J_i(su_1 + (1 - s)u_2),
	\]
	so $\mathcal J_i$ is strictly convex.
	
	The maps $u \mapsto (\square + q_i) u$ and $u \mapsto \partial_\nu u$ are continuous linear from $H^2(Q_T)$ to $L^2(Q_T)$ and to $L^2(\Sigma_T)$, respectively. And the weighted $L^2$ norm in \eqref{eq:sFun-PR25} is weakly lower semicontinuous, so in total $\mathcal J_i$ is weakly lower semicontinuous.
	Moreover, $\mathcal J_i(u) \geq 0$ for any $u \in H^2(Q_T)$.
	Therefore, by Lemma \ref{lem:mini-PR25} we can conclude $\mathcal J_i$ attains its minimum on $\mathcal A$ at a unique point, and by \eqref{eq:sUpd-PR25} this point is $u_i$.
	The proof is done.
\end{proof}

We have proved the existence of each $u_i$, and hence of each $q_i$, in \eqref{eq:sUpd-PR25}.
The convergence result $\lim_{i \to \infty} q_i = q^*$ is deferred to Section \ref{subsec:sCon-PR25}.
Before that, we examine how each $u_i$ is constructed.
For each $u_i$ we shall construct sequence $\{u_i^{(j)}\}$ such that $u_i^{(j)} \rightharpoonup u_i$ and $T_M ( u_i^{(j)}(0,\cdot) ) \to q_{i+1}$.

\begin{prop} \label{prop:pg-PR25}
	Let $A_i u := (\square + q_i)u$ and $Bu := \partial_\nu u|_{\Sigma_T}$, where the norms are the operator norms from $H^2(Q_T)$ to $L^2(Q_T, e^{2\lambda\varphi})$ and $L^2(\Sigma_T, e^{2\lambda\varphi})$, respectively.
	Then $\mathcal J_i$ is of class $C^{1,1}$ on $H^2(Q_T)$, with
	\begin{equation} \label{eq:grJ-PR25}
		\mathcal J_i'(u) = 2A_i^*(A_i u) + 2\lambda B^*(Bu - \partial_\nu u^*),
	\end{equation}
	where the gradient and the adjoints are taken with respect to the $H^2(Q_T)$ inner product.
	The projected gradient iteration
	\begin{equation} \label{eq:pg-PR25}
		u_i^{(j + 1)} = \mathcal P_{\mathcal A} \big( u_i^{(j)} - \tau \mathcal J_i'(u_i^{(j)}) \big),
		\qquad 0 < \tau < (\norm{A_i}^2 + \lambda \norm{B}^2)^{-1},
	\end{equation}
	with $\mathcal P_{\mathcal A}$ the metric projection onto $\mathcal A$ in $H^2(Q_T)$, starting from any $u_i^{(0)} \in \mathcal A$, satisfies
	\begin{equation} \label{eq:Jiui-PR25}
		\mathcal J_i(u_i^{(j)}) \to \min_{u \in \mathcal A} \mathcal J_i(u) \ \ \mbox{and} \ \
		u_i^{(j)} \rightharpoonup u_i \ \ \text{in} \ \ H^2(Q_T), \quad \mbox{as} \quad j \to \infty,
	\end{equation}
	where $u_i$ is defined in \eqref{eq:sUpd-PR25}.
	Consequently,
	\begin{equation} \label{eq:uq-PR25}
		q_{i+1} = \lim_{j \to \infty} T_M \big( u_i^{(j)}(0,\cdot) \big) \ \ \mbox{in} \ \ L^2(\Omega).
	\end{equation}
\end{prop}

\begin{proof}
	Denote $\mathcal H := H^2(Q_T)$, $\mathcal Y := L^2(Q_T, e^{2\lambda\varphi})$ and $\mathcal Z = L^2(\Sigma_T, e^{2\lambda\varphi})$ for short.
	Here $\mathcal Y$ consists of all measurable $f$ with $\int_{Q_T} e^{2\lambda\varphi} f^2 \dif x \dif t < \infty$, and $\mathcal Z$ is defined analogously.
	Then we can rewrite $\mathcal J_i(u)$ as
	\[
	\mathcal J_i(u) = \norm{A_i u}_{\mathcal Y}^2 + \lambda \norm{Bu - \partial_\nu u^*}_{\mathcal Z}^2,
	\qquad u \in \mathcal H.
	\]
	
	First, we show that $\mathcal J_i$ is of class $C^{1,1}$ on $\mathcal H$.
	For $u, v \in \mathcal H$, a direct calculation gives
	\begin{equation*}
		\mathcal J_i(u+v) - \mathcal J_i(u)
		= 2\langle A_i u, A_i v\rangle_{\mathcal Y}
		+ 2\lambda\langle Bu - \partial_\nu u^*, Bv \rangle_{\mathcal Z} + \norm{A_i v}_{\mathcal Y}^2 + \lambda \norm{Bv}_{\mathcal Z}^2.
	\end{equation*}
	Note that $\lambda$ is large but fixed.
	Since $\varphi$ is bounded on $Q_T$ and on $\Sigma_T$, the weighted norms are equivalent to the corresponding unweighted $L^2$-norms.
	Hence the maps $A_i : \mathcal H \to \mathcal Y$ and $B : \mathcal H \to \mathcal Z$ are bounded linear operators, thus
	\[
	\norm{A_i v}_{\mathcal Y}^2 + \lambda \norm{Bv}_{\mathcal Z}^2
	\leq (\norm{A_i}^2 + \lambda \norm{B}^2) \norm{v}_{\mathcal H}^2
	= o(\norm{v}_{\mathcal H}).
	\]
	This, together with the expansion of $\mathcal J_i(u+v) - \mathcal J_i(u)$, shows $\mathcal J_i$ is Fr\'echet differentiable, and its derivative $D \mathcal J_i$ at $u$ is given by
	\begin{align}
		D\mathcal J_i(u) v
		:= & \, 2\langle A_i u, A_i v\rangle_{\mathcal Y}
		+ 2\lambda \langle Bu - \partial_\nu u^*, Bv\rangle_{\mathcal Z} \nonumber \\
		= & \, \langle 2A_i^*(A_i u) + 2\lambda B^*(Bu - \partial_\nu u^*),\, v \rangle_{\mathcal H}, \label{eq:DJi-PR25}
	\end{align}
	where $A_i^* : \mathcal Y \to \mathcal H$ and $B^* : \mathcal Z \to \mathcal H$ are the adjoint operators of $A_i$ and $B$, respectively.
	Denote
	\begin{equation} \label{eq:naJi-PR25}
		\mathcal J_i'(u)
		= 2A_i^*(A_i u) + 2\lambda B^*(Bu - \partial_\nu u^*).
	\end{equation}
	With the help of $\mathcal J_i'(u)$ and the boundedness of $A_i$ and $B$, for $u, v \in \mathcal H$ we can have
	\begin{align}
		\mathcal J_i(u+v)
		& = \mathcal J_i(u) + \langle \mathcal J_i'(u), v \rangle_{\mathcal H} + \norm{A_i v}_{\mathcal Y}^2 + \lambda \norm{Bv}_{\mathcal Z}^2 \label{eq:Jiu1-PR25} \\
		& \leq \mathcal J_i(u) + \langle \mathcal J_i'(u), v \rangle_{\mathcal H} + (\norm{A_i}^2 + \lambda \norm{B}^2) \norm{v}_{\mathcal H}^2, \label{eq:Jiuv-PR25}
	\end{align}
	where the norms on $A_i$ and $B$ are operator norms.
	Moreover, we have
	\begin{align*}
		\norm{\mathcal J_i'(u) - \mathcal J_i'(v)}_{\mathcal H}
		& = \norm{2A_i^*A_i(u-v) + 2\lambda B^*B(u-v)}_{\mathcal H} \\
		& \leq 2(\norm{A_i^*A_i} + \lambda \norm{B^*B}) \norm{u-v}_{\mathcal H} \\
		& \leq 2(\norm{A_i}^2 + \lambda \norm{B}^2) \norm{u-v}_{\mathcal H}.
	\end{align*}
	Hence, we have proved $\mathcal J_i \in C^{1,1}(\mathcal H)$ with Lipschitz constant $2(\norm{A_i}^2 + \lambda \norm{B}^2)$.
	
	Now we want to use Lemma \ref{lem:app-PR25} to draw the conclusion, so we check the conditions one-by-one.
	It can be checked that the set $\mathcal A$ defined in \eqref{eq:sAdm-PR25} is nonempty, bounded, closed and convex, so by Lemma \ref{lem:exst-PR25}, the functional $\mathcal J_i$ has a unique minimizer $u_i \in \mathcal A$.
	By Lemma \ref{lem:exst-PR25}, $\mathcal J_i$ is convex on $H^2(Q_T)$ and strictly convex on $\mathcal A$.
	As mentioned in the proof of Lemma \ref{lem:exst-PR25}, $\mathcal J_i$ is weakly lower semicontinuous.
	Moreover, $\mathcal J_i \geq 0$ on $\mathcal A$, and the number $\tau$ in \eqref{eq:pg-PR25} satisfies $0 < \tau < (\norm{A_i}^2 + \lambda \norm{B}^2)^{-1}$.
	Hence, by Lemma \ref{lem:app-PR25}, for any starting point $u_i^{(0)} \in \mathcal A$, the sequence $\{ u_i^{(j)}
	\}_{j \geq 0}$ defined in \eqref{eq:pg-PR25} satisfies
	\[
	\mathcal J_i(u_i^{(j)}) \to \min_{u \in \mathcal A} \mathcal J_i(u) \ \ \mbox{and} \ \
	u_i^{(j)} \rightharpoonup u_i \ \ \text{in} \ \ H^2(Q_T), \quad \mbox{as} \quad j \to \infty.
	\]
	As for the starting points, for $i \geq 1$ we take $u_i^{(0)} = u_{i-1} \in \mathcal A$, and for $i = 0$ we take $u_0^{(0)} = \mathcal P_{\mathcal A}(0)$, the metric projection of $0$ onto $\mathcal A$.
	The particular choice of starting point does not affect the limit in \eqref{eq:Jiui-PR25}, since Lemma \ref{lem:app-PR25} applies to every point of $\mathcal A$, and $\mathcal J_i$ is strictly convex on $\mathcal A$, so its minimizer is unique.
	
	Finally, the trace map $u \mapsto u(0,\cdot)$ is compact from $H^2(Q_T)$ to $L^2(\Omega)$, so the weak convergence implies the strong convergence $u_i^{(j)}(0,\cdot) \to u_i(0,\cdot)$ in $L^2(\Omega)$.
	Since $T_M$ is $1$-Lipschitz, we obtain
	\begin{align*}
		\norm{T_M(u_i^{(j)}(0,\cdot)) - q_{i+1}}_{L^2(\Omega)}^2
		& = \int_\Omega \big| T_M \big( u_i^{(j)}(0,x) \big) - T_M \big( u_i(0,x) \big) \big|^2 \dif x \\
		& \leq \norm{u_i^{(j)}(0,\cdot) - u_i(0,\cdot)}_{L^2(\Omega)}^2,
	\end{align*}
	which implies \eqref{eq:uq-PR25}.
	The proof is done.
\end{proof}

\subsection{Geometric convergence} \label{subsec:sCon-PR25}

We are ready to give the geometric convergence of the iterative reconstruction and to prove Theorem \ref{thm:3-PR25}.

\begin{prop} \label{prop:cntr-PR25}
	Let $u_i$ and $q_{i + 1}$ be defined by \eqref{eq:sUpd-PR25}, and let $M > 0$ be such that $\norm{g^*}_{H^{s+1/2}(\Sigma_T)} + \norm{q^*}_{C^{s+1}(\overline \Omega)} \leq M$.
	Set $\varphi$ as in Proposition \ref{prop:InCa-PR25}, then there exist positive constants $\lambda_0 = \lambda_0 (\Omega,T,\eta,M)$ and $C = C(\beta,T,\Omega,\lambda_0,s,M)$ such that for all $\lambda \geq \lambda_0$ and $i \geq 0$, there holds
	\begin{equation} \label{eq:sCon-PR25}
		\int_{\Omega_0} e^{2\lambda \varphi} (q_{i+1} - q^*)^2 \dif x
		\leq C \beta^{-1/2} \lambda^{-5/2} \int_\Omega e^{2\lambda\varphi} (q_i - q^*)^2 \dif x.
	\end{equation}
	Consequently, when $\lambda > \max \{ \lambda_0, C\beta^{-1/5} \}$, the weighted $L^2$-distance between $q_i$ and $q^*$ tends to zero geometrically.
	Moreover, $q^*$ is the unique fixed point of the iteration in the set $\{q \in L^\infty(\Omega) \,;\, \norm{q}_{L^\infty(\Omega)} \leq M \}$.
\end{prop}

\begin{proof}
	Since $\norm{q_i}_{L^\infty(\Omega)} \leq M$ by \eqref{eq:sUpd-PR25}, the constants in Proposition \ref{prop:enCr-PR25} are independent of $i$.
	We know $(\square + q^*) u^* = 0$, then we have $(u_i - u^*)|_{\Sigma_T} = 0$, $\partial_t (u_i - u^*) (0, \cdot) = 0$, and
	\[
	(\square + q_i) (u_i - u^*) = (\square + q_i) u_i - (q_i - q^*) u^*.
	\]
	Applying Proposition \ref{prop:enCr-PR25} to $u_i - u^*$ gives
	\begin{align*}
		C \lambda^2 \int_{\Omega_0} e^{2\lambda \varphi} (u_i - u^*)^2
		& \leq \int_{Q_T} e^{2\lambda \varphi} |(\square + q_i) u_i - (q_i - q^*) u^*|^2 + \lambda \int_{\Sigma_T} e^{2\lambda \varphi} (\partial_\nu u_i - \partial_\nu u^*)^2 \\
		& \leq 2\int_{Q_T} e^{2\lambda \varphi} |(q_i - q^*) u^*|^2 + 2 \mathcal J_i(u_i),
	\end{align*}
	where the term $\mathcal B$ in \eqref{eq:WeC1-PR25} vanishes because $\partial_t (u_i - u^*)(0,\cdot) = 0$ in $\Omega$.
	We have $u^* \in \mathcal A$ and $(\square + q_i) u^* = (q_i - q^*) u^*$.
	By minimality, we know $\mathcal J_i(u_i) \leq \mathcal J_i(u^*)$, thus we can continue the calculation as
	\begin{align*}
		C \lambda^2 \int_{\Omega_0} e^{2\lambda \varphi} (u_i - u^*)^2 \dif x
		& \leq 2\int_{Q_T} e^{2\lambda \varphi} |(q_i - q^*) u^*|^2 \dif x \dif t + 2 \mathcal J_i(u^*) \\
		& = 2 \int_{Q_T} e^{2\lambda \varphi} |(q_i - q^*) u^*|^2 \dif x \dif t + 2 \int_{Q_T} e^{2\lambda \varphi} |(\square + q_i) u^*|^2 \dif x \dif t \\
		& = 4 \int_{Q_T} e^{2\lambda \varphi} |(q_i - q^*) u^*|^2 \dif x \dif t.
	\end{align*}
	Because $\norm{g^*}_{H^{s+1/2}(\Sigma_T)} + \norm{q^*}_{C^{s+1}(\overline \Omega)} \leq M$, by Corollary \ref{cor:Linf-PR25} we have
	\begin{equation*}
		\norm{u^*}_{L^\infty(Q_T)}
		\leq C(T,\Omega,s) \,
		e^{C ( 1 + \norm{q^*}_{C^{s+1}} ) T} (\norm{g^*}_{H^{s+1/2}(\Sigma_T)} + \norm{q^*}_{H^{s+1}(\Omega)})
		\leq C(T,\Omega,s,M),
	\end{equation*}
	thus
	\begin{align}
		C \lambda^2 \int_{\Omega_0} e^{2\lambda \varphi} (u_i - u^*)^2 \dif x
		& \leq 4 C(T,\Omega,s,M) \int_{Q_T} e^{2\lambda\varphi} (q_i - q^*)^2 \dif x \dif t \nonumber \\
		& \leq 2 C(T,\Omega,s,M) \sqrt {\frac \pi {\lambda \beta}} \int_\Omega e^{2\lambda\varphi} (q_i - q^*)^2 \dif x. \label{eq:sMi1-PR25}
	\end{align}
	
	Since $\norm{q^*}_{L^\infty(\Omega)} \leq \norm{q^*}_{C^{s+1}(\overline \Omega)} \leq M$ by \eqref{eq:gpM-PR25}, we have $q^* = T_M(q^*)$, so
	\begin{equation} \label{eq:sMi2-PR25}
		|q_{i+1} - q^*|
		= |T_M(u_i(0,\cdot)) - T_M(q^*)|
		\leq |u_i(0,\cdot) - q^*|
		= |u_i(0,\cdot) - u^*(0,\cdot)|.
	\end{equation}
	Combining \eqref{eq:sMi1-PR25} and \eqref{eq:sMi2-PR25} gives \eqref{eq:sCon-PR25}.
	The proof is complete.
\end{proof}

The following result is immediate.

\begin{cor} \label{cor:L2-PR25}
	Under the hypotheses of Proposition \ref{prop:cntr-PR25} and $q_0 = 0$, for every $i \geq 0$, there holds
	\begin{equation*}
		\norm{q_i - q^*}_{L^2(\Omega)}
		\leq M |\Omega|^{1/2} e^{\lambda [\sup_{\Omega} \varphi(0,\cdot) - \inf_{\Omega} \varphi(0,\cdot)]} (C \beta^{-1/4} \lambda^{-5/4})^i.
	\end{equation*}
\end{cor}

\begin{proof}
	From \eqref{eq:sCon-PR25}, we have
	\begin{equation*}
		\Big( \int_{\Omega_0} e^{2\lambda \varphi} (q_i - q^*)^2 \dif x \Big)^{1/2}
		\leq (C \beta^{-1/4} \lambda^{-5/4})^i \Big( \int_\Omega e^{2\lambda\varphi} (q_0 - q^*)^2 \dif x \Big)^{1/2},
	\end{equation*}
	which, together with $q_0 = 0$ and $\norm{q^*}_{L^\infty(\Omega)} \leq M$, implies the result.
\end{proof}

Now Theorem \ref{thm:3-PR25} follows naturally.

\begin{proof}[Proof of Theorem \ref{thm:3-PR25}]
	The first limit in \eqref{eq:qij-PR25} comes from \eqref{eq:uq-PR25} in Proposition \ref{prop:pg-PR25}, and the second limit comes from Corollary \ref{cor:L2-PR25}.
\end{proof}

\section*{Acknowledgements}

The research of S.M.~is partially supported by the NSFC (No.~12301540).
The author would like to thank Professor Mikhail Klibanov for reading the draft version of this paper and for giving helpful comments.

\vspace{.3cm}
\noindent \textbf{Data Availability:} Data sharing is not applicable to this article as no new data were created or analyzed in this study.

\vspace{.3cm}
\noindent \textbf{Conflict of Interest:} The author declares that there is no conflict of interest regarding the publication of this paper.

%\bibliography{reference_bib_PR2025}
%\bibliographystyle{alpha}

{

% \bib, bibdiv, biblist are defined by the amsrefs package.

}

\end{document}